\documentclass[a4paper,reqno]{amsart}
\usepackage[utf8]{inputenc}
\usepackage[english]{babel}

\usepackage{amsfonts,amsthm,amssymb,mathtools}
\usepackage[alphabetic,abbrev]{amsrefs}
\numberwithin{equation}{section}

\usepackage{hyperref}
\usepackage[capitalise,noabbrev]{cleveref} 
\crefname{equation}{}{}
\Crefname{equation}{}{}

\newtheorem{theorem}{Theorem}[section]
\newtheorem{lemma}[theorem]{Lemma}
\newtheorem{proposition}[theorem]{Proposition}
\newtheorem{corollary}[theorem]{Corollary}

\theoremstyle{definition}

\theoremstyle{remark}
\newtheorem*{remark}{Remark}

\newcommand{\N}{\mathbb{N}}
\newcommand{\Z}{\mathbb{Z}}

\newcommand{\R}{\mathbb{R}}
\newcommand{\C}{\mathbb{C}}
\renewcommand{\S}{\mathbb{S}}

\newcommand{\g}{\mathfrak{g}}

\newcommand{\G}{\mathcal{G}}

\renewcommand{\S}{\mathcal{S}}

\DeclarePairedDelimiter\norm{\lVert}{\rVert}

\renewcommand{\Im}{\operatorname{Im}}

\DeclareMathOperator{\supp}{supp}
\DeclareMathOperator{\id}{id}
\newcommand{\sloc}{\mathrm{sloc}}

\AtBeginDocument{%
   \def\MR#1{}
}

\begin{document}

\title[Spectral multipliers on Heisenberg type groups]
{An $L^p$-spectral multiplier theorem with sharp $p$-specific regularity bound on Heisenberg type groups}
\author{Lars Niedorf}
\address{Mathematisches Seminar, Christian-Albrechts-Universität zu Kiel, Heinrich-Hecht-Platz 6, 24118 Kiel, Germany}
\email{niedorf@math.uni-kiel.de}
\date{April 9, 2024}
\thanks{This version of the article has been accepted for publication, after peer review but is not the Version of Record and does not reflect post-acceptance improvements, or any corrections. The Version of Record is available online at: \url{http://dx.doi.org/10.1007/s00041-024-10075-1}.\\
The author gratefully acknowledges the support by the Deutsche Forschungsgemeinschaft (DFG) through grant MU 761/12-1.}

\maketitle

\begin{abstract}
We prove an $L^p$-spectral multiplier theorem for sub-Laplacians on Heisenberg type groups under the sharp regularity condition $s>d\left|1/p-1/2\right|$, where $d$ is the topological dimension of the underlying group. Our approach relies on restriction type estimates where the multiplier is additionally truncated along the spectrum of the Laplacian on the center of the group.
\end{abstract}

\section{Introduction}

\subsection{Statement of the main result}

Let $G$ be a two-step stratified Lie group, that is, a connected, simply connected, two-step nilpotent Lie group whose Lie algebra~$\g$ admits a decomposition $\g=\g_1\oplus \g_2$ with $[\g_1,\g_1]=  \g_2$ and $\g_2\subseteq \g$ being contained in the center of $\g$. Let $X_1,\dots,X_{d_1}$ be a basis of $\g_1$. This basis can be identified with a system of left-invariant vector fields on $G$ via the Lie derivative. Let $L$ be the sub-Laplacian associated with the vector fields $X_1,\dots,X_{d_1}$, that is, the second order, left-invariant differential operator given by
\[
L=-(X_1^2+\dots + X_{d_1}^2).
\]
This operator is positive and self-adjoint on $L^2(G)$, where $G$ is endowed with a left-invariant Haar measure. Via functional calculus, one can define for every Borel measurable function $F:\R\to \C$ the operator $F(L)$ on $L^2(G)$, which is a bounded operator on $L^2(G)$ whenever the spectral multiplier $F$ is bounded. Regarding multipliers $F$ for which $F(L)$ extends to a bounded operator on $L^p(G)$, sufficient conditions can be given in terms of differentiability properties of the multiplier $F$, usually expressed by the scale-invariant localized Sobolev norms $\Vert\cdot\Vert_{L^2_{s,\sloc}}$, $s\ge 0$, given by
\[
\norm{F}_{L^2_{s,\sloc}} := \sup_{t>0}{ \norm{ F(t\,\cdot\,) \eta }_{L_s^2(\R)} }.
\]
Here, $\eta:\R\to\C$ is a smooth non-zero function with compact support in $(0,\infty)$ and $L_s^2(\R)\subseteq L^2(\R)$ denotes the Sobolev space of (fractional) order $s\ge 0$. Due to a celebrated theorem of Christ \cite{Ch91} and of Mauceri and Meda \cite{MaMe90}, $F(L)$ extends to a bounded operator on all $L^p$-spaces for $1<p<\infty$ whenever
\[
\norm{F}_{L^2_{s,\sloc}} < \infty\quad \text{for some } s>Q/2,
\]
where $Q=\dim \g_1+2\dim \g_2$ is the homogeneous dimension of the underlying Lie group. Moreover, $F(L)$ is of weak type $(1,1)$, i.e., bounded from $L^1(G)$ to the Lorentz space $L^{1,\infty}(G)$. (Actually, the theorem holds true for stratified Lie groups of arbitrary step, but our focus lies on stratified Lie groups of step two.)

In the case of Heisenberg (-type) groups, Müller and Stein \cite{MueSt94}, and independently Hebisch \cite{He93} showed that the threshold $s>Q/2$ can even be pushed down to $s > d/2$, where $d$ is the topological dimension of the underlying group. This result has been extended to other specific classes of two-step stratified Lie groups \cite{Ma12,MaMue14b,Ma15} (and also to other settings, cf.\ \cite{CoSi01,MaMue14a,CaCoMaSi17,AhCoMaMue20}), but up to now, it is still open whether the threshold $s>d/2$ is sufficient for any two-step stratified Lie group. However, Martini and Müller \cite{MaMue16} were able to show that for all two-step stratified Lie groups and left-invariant sub-Laplacians, the sharp threshold is strictly less than $Q/2$, but not less than $d/2$.

On the other hand, instead of asking for boundedness on all $L^p$-spaces for $1<p<\infty$ simultaneously, one can ask for the minimal threshold $s_p\in [0,d/2]$ such that $F(L)$ is bounded on $L^p$ whenever $\norm{F}_{L^2_{s,\sloc}} < \infty$ for some $s>s_p$. In \cite{MaMueNi19}, Martini, Müller, and Nicolussi Golo showed for a large class of smooth second-order real differential operators associated with a sub-Riemannian structure on smooth $d$-dimensional manifolds that at least regularity of order $s \ge d\left |1/p-1/2\right|$ is necessary for having $L^p$-spectral multiplier estimates. On the opposite, one expects this threshold also to be essentially sufficient. Sufficiency results featuring the regularity condition $s>Q\left|1/p-1/2\right|$, where $Q\ge d$ is the homogeneous dimension, are available in various settings \cite{ChOuSiYa16,LeRoSe14,Se86}, but beyond the Euclidean setting (where $d=Q$), to the best of my knowledge, sufficiency of the threshold $s \ge d\left |1/p-1/2\right|$ has so far only been proven in exceptional cases, see \cite{ChOu16,Ni21}.

The purpose of this paper is to extend the results for Grushin operators of \cite{Ni21} to sub-Laplacians on Heisenberg type groups, which is the following subclass of two-step stratified Lie groups:  Suppose that $\g$ is endowed with an inner product $\langle \cdot,\cdot\rangle$ for which the stratification $\g=\g_1\oplus\g_2$ is orthogonal. Let $\g_2^*$ denote the dual of $\g_2$. For any $\mu\in \g_2^*$, we have a skew-symmetric bilinear form $\omega_\mu : \g_1\times \g_1\to \R$ given by
\begin{equation}\label{eq:symplectic-form}
\omega_\mu(x,x'):=\mu([x,x']),\quad x,x'\in\g_1.
\end{equation}
Then, for any $\mu\in\g_2^*$, there is a skew-symmetric endomorphism $J_\mu$ on $\g_1$ such that $\omega_\mu(x,x') = \langle J_\mu x,x' \rangle$ for all $x,x'\in\g_1$. Let $|\cdot|$ denote the norm on $\g_2^*$ induced by the inner product $\langle\cdot,\cdot\rangle$. Then the group $G$ is called a \textit{Heisenberg type group} if $J_\mu$ is orthogonal for $|\mu|=1$, that is,
\[
J_\mu^2 = - |\mu|^2\id_{\g_1} \quad\text{for all } \mu\in\g_2^*.
\]
Note that this implies in particular that $\g_2=[\g_1,\g_1]$ is the center of $\g$.
Given any orthonormal basis $X_1,\dots,X_{d_1}$ of $\g_1$ with respect to the inner product $\langle\cdot,\cdot\rangle$, we consider the associated sub-Laplacian
\begin{equation}\label{def:sub-Laplace}
L=-(X_1^2+\dots + X_{d_1}^2).
\end{equation}
Our main result is the following spectral multiplier estimate, together with a corresponding result for Bochner--Riesz multipliers. By the result of \cite{MaMueNi19}, the threshold $s > d \left(1/p - 1/2\right)$ is optimal up to the endpoint and cannot be decreased. Note that the required order of regularity in the second part of \cref{thm:main} is the same as in the Bochner--Riesz conjecture, see for example \cite{Ta99}. However, note also that \cref{thm:main} only makes a statement in the range $1\le p \le 2(d_2+1)/(d_2+3)$, where $d_2$ is the dimension of the center of the group. 

\begin{theorem}\label{thm:main}
Let $G$ be a Heisenberg type group of topological dimension $d=\dim G$ and center of dimension $d_2$, and let $L$ be a sub-Laplacian as in \cref{def:sub-Laplace}. Suppose that $1\le p \le 2(d_2+1)/(d_2+3)$. Then the following statements hold:
\begin{enumerate}
\item \label{main-(1)}%
If $p>1$ and if $F:\R\to \C$ is a bounded Borel function such that
\[
\norm{F}_{L^2_{s,\sloc}} < \infty
\quad\text{for some } s > d \left(1/p - 1/2\right),
\]
then the operator $F(L)$ is bounded on $L^p(G)$, and
\[
\norm{F(L)}_{L^p\to L^p} \le C_{p,s} \norm{F}_{L^2_{s,\sloc}}.
\]
\item \label{main-(2)}%
For any $\delta> d \left(1/p - 1/2\right)- 1/2$, the Bochner--Riesz means $(1-tL)^\delta_+$, $t\ge 0$, are uniformly bounded on $L^p(G)$.
\end{enumerate}
\end{theorem}

If $s>d/2$ in the first part of \cref{thm:main}, then the operator $F(L)$ is of weak type $(1,1)$ by \cite{He93}. The condition $1\le p \le 2(d_2+1)/(d_2+3)$ in \cref{thm:main} derives from the Stein-Tomas restriction estimate on $\g_2$, which is used for the proof of the spectral multiplier estimates. Due to this assumption, the first part of the theorem only gives results when $d_2\ge 2$, which means that $G$ must be a Heisenberg type group which is not a Heisenberg group. This reflects the phenomenon that there are no good analogues of Fourier restriction estimates available in the case of Heisenberg groups, except for the trivial $L^1$-$L^\infty$ estimates, which is due to the fact that the Heisenberg group admits only a one-dimensional center \cite{Mue90}. It remains an open question whether such $p$-specific spectral multiplier estimates as above also hold in the setting of Heisenberg groups.

The spectral multiplier results of \cref{thm:main} are extended in a follow-up paper \cite{Nie23b} to the class of Métivier groups. However, the approach of \cite{Nie23b} requires additional methods due to the fact that the matrices $J_\mu$ from above are no longer orthogonal and the spectral decomposition into eigenspaces, which depend on $\mu$, is more complicated.

\subsection{Structure of the proof}

Building on methods of \cite{CaCi13,ChOu16,Ni21}, the proof of the spectral multiplier estimates of \cref{thm:main} relies on restriction type estimates, a fundamental connection that was first discovered by C.\ Fefferman \cite{Fe73} and has since then been exploited by many other authors, see \cite{SeSo89,GuHaSi13,ChOuSiYa16}. The key idea can be illustrated as follows. Suppose that we want to derive $L^p$-boundedness of the Bochner--Riesz means $(1-L)_+^\delta$. If $F:\R\to\C$ is the multiplier given by
\[
F(\lambda)=(1-\lambda^2)^\delta_+,\quad \lambda\in\R,
\]
then $F(\sqrt L)=(1-L)_+^\delta$. We decompose $F$ as
\[
F = \sum_{\iota = 0}^\infty F^{(\iota)},
\]
where $|\tau|\sim 2^\iota$ whenever $\tau \in\supp \widehat{F^{(\iota)}}$ for $\iota\ge 1$. It suffices to show
\[
\norm{F^{(\iota)}(\sqrt L)}_{p\to p}\lesssim 2^{-\varepsilon\iota}
\quad\text{for some }\varepsilon>0.
\]
Let $\mathcal K^{(\iota)}$ be the convolution kernel associated with $F^{(\iota)}(\sqrt L)$, that is,
\[
F^{(\iota)}(\sqrt L) f = f * \mathcal K^{(\iota)} \quad\text{for } f\in\mathcal S(G),
\]
where $*$ denotes the group convolution. By the Fourier inversion formula, we have
\[
F^{(\iota)} ( \sqrt L) = \frac{1}{2\pi} \int_{|\tau|\sim 2^\iota} \chi(\tau/2^\iota) \hat F(\tau) \cos(\tau \sqrt L) \,d\tau,
\]
where $\chi:\R\to\C$ is some smooth function which is compactly supported in $(0,\infty)$. Thanks the frequency localization of the dyadic pieces $F^{(\iota)}$, since solutions of the wave equation associated with $L$ possess finite propagation speed, the convolution kernel $\mathcal K^{(\iota)}$ is supported in a Euclidean ball of dimension $B_R \times B_{R^2}\subseteq \R^{d_1}\times \R^{d_2}$ centered at the origin with radius $R\sim 2^\iota$. Now having a restriction type estimate which is essentially of the form
\begin{equation}\label{intro-restriction-1}
\norm{F^{(\iota)}(\sqrt L)f}_2  \lesssim \norm{F^{(\iota)}}_2 \norm{f}_p,
\end{equation}
then, given a function $f$ with $\supp f\subseteq B_R \times B_{R^2}$ (which one can assume without loss of generality), Hölder's inequality provides
\begin{equation}\label{eq:intro-hoelder-1}
\norm{F^{(\iota)}(\sqrt L)f}_p \lesssim 
 R^{Q/q} \norm{F^{(\iota)}(\sqrt L)f}_2
 \lesssim R^{Q/q} \norm{F^{(\iota)}}_2 \norm{f}_p,
\end{equation}
where $Q=d_1+2d_2$ and $1/q=1/p-1/2$. On the other hand, one can show that $F\in L^2_{Q/q+\varepsilon}(\R)$ for some $\varepsilon>0$ if $\delta>Q/q-1/2$, whence the last term of \cref{eq:intro-hoelder-1} can be estimated via
\[
R^{Q/q} \norm{F^{(\iota)}}_2 \norm{f}_p
 \lesssim 2^{-\varepsilon\iota} \norm{F^{(\iota)}}_{L^2_{Q/q+\varepsilon}(\R)} \norm{f}_p.
\]
However, this approach does in general only provide thresholds featuring the homogeneous dimension $Q$ in place of the topological one. Employing the approach of \cite{Ni21}, we additionally decompose the multipliers $F^{(\iota)}$ dyadically along the spectrum of the Laplacian
\[
U = (-(U_1^2+\dots+U_{d_2}^2))^{1/2},
\]
where $U_1,\dots,U_{d_2}$ is a basis of the second layer $\g_2$ of the Lie algebra $\g$. More precisely, we decompose $F^{(\iota)}$ such that
\begin{equation}\label{eq:intro-truncation}
F^{(\iota)}(\sqrt L) = \sum_{\ell = 0}^\infty F^{(\iota)}(\sqrt L) \chi_\ell(L/U),
\end{equation}
where $(\chi_\ell)_{\ell\in\Z}$ is a dyadic decomposition of $\R\setminus\{0\}$ such that $|\lambda|\sim 2^\ell$ if $\lambda\in\supp \chi_\ell$. The convolution kernel $\mathcal K^{(\iota)}_\ell$ of the operator $F_\ell^{(\iota)}(L,U) = F^{(\iota)}(\sqrt L) \chi_\ell(L/U)$ can be explicitly written down in terms of the Fourier transform and rescaled Laguerre functions $\varphi_k^{|\mu|}$ (for their definition, see \cref{eq:laguerre-function-1} below), namely,
\begin{align}
\mathcal K^{(\iota)}_\ell(x,u) = (2\pi)^{-d_2} \sum_{k=0}^\infty \int_{\g_2^*\setminus\{0\}} & F^{(\iota)}\big(\sqrt{[k]|\mu|}\big) \chi_\ell([k]) \varphi_k^{|\mu|}(x) e^{i\langle \mu, u\rangle} \, d\mu, \label{eq:intro-conv-kernel}
\end{align}
where $[k]:=2k+d_1/2$. Thus, due to the factor $\chi_\ell([k])$ in the integral, the truncation achieved by the function $\chi_\ell$ corresponds to taking in \cref{eq:intro-conv-kernel} only the summands with $[k]\sim 2^\ell$. Assuming in the following that $F^{(\iota)}$ is compactly supported around $1$ (which is of course in general not true in view of the Paley--Wiener theorem but can be achieved by a cut-off), we have $|\mu|\sim [k]^{-1}$ and thus $|\mu|\sim 2^{-\ell}$ on the support of $F_\ell^{(\iota)}$, whence the truncation afforded by the functions $\chi_\ell$ is also referred to as a \textit{truncation along the spectrum of $U$}. Pointwise estimates for Laguerre functions (cf.\ Eq.\ (1.1.44)\footnote{There is a small typo in Eq.\ (1.1.44): The factor $e^{-x}$ has to be replaced by $e^{-x/2}$.}, (1.3.41), and Lemma~1.5.3 of \cite{Th93} or alternatively the table on page 699 of \cite{AsWa65}) suggest that $\varphi_k(x):=\varphi_k^1(x)$ has exponential decay for $|x|\gtrsim [k]^{1/2}$. Hence, since $|\mu|\sim [k]^{-1}$, the function $\varphi_k^{|\mu|}$, which satisfies
\[
\varphi_k^{|\mu|}(x)=|\mu|^{d_1/2}\varphi_k(|\mu|^{1/2}x),
\]
is supported where $|x|\lesssim |\mu|^{-1/2} [k]^{1/2}\sim 2^\ell=:R_\ell$, up to some exponentially decaying term. This means that the kernel $\mathcal K^{(\iota)}_\ell$ is \textit{essentially} supported in a Euclidean ball of dimension $B_{R_\ell}\times B_{R^2}\subseteq \R^{d_1}\times \R^{d_2}$ centered at the origin. Now, instead of \cref{intro-restriction-1}, suppose we had a restriction type estimate of the form
\begin{equation}\label{eq:intro-restriction-2}
\norm{F_\ell^{(\iota)}(L,U)f}_2  \lesssim R_\ell^{-d_2/q} \norm{F^{(\iota)}}_2 \norm{f}_p.
\end{equation}
We distinguish the cases $0\le \ell\le \iota$ and $\ell>\iota$. In the former case, we decompose the ball $B_R \times B_{R^2}$ into a grid of balls $B_m^{(\ell)}$ of dimension $R_\ell \times R^2$ with respect to the Euclidean distance on the layers $\g_1$ and $\g_2$, respectively. Correspondingly, we decompose the function $f$ supported in $B_R \times B_{R^2}$ into a sum of functions $f_m$ supported in balls of dimension $R_\ell \times R^2$. Then Hölder's inequality, applied with $1/p=1/q+1/2$, combined with the restriction type estimate \cref{eq:intro-restriction-2} would provide
\begin{align}
\norm{F_\ell^{(\iota)}(L,U)f_m}_p
& \lesssim 
 (R_\ell^{d_1} R^{2d_2})^{1/q} \norm{F_\ell^{(\iota)}(L,U)f_m}_2 \notag \\
& \lesssim (R_\ell^{d_1-d_2} R^{2d_2})^{1/q} \norm{F_\ell^{(\iota)}}_2 \norm{f_m}_p.
\label{eq:intro-hoelder-2}
\end{align}
Since $G$ is a Heisenberg type group, we have $d_2<d_1$. Together with $\ell\le \iota$, we obtain
\[
R_\ell^{d_1-d_2} R^{2d_2} \sim 2^{(\ell-\iota)(d_1-d_2)} 2^{\iota(d_1+d_2)} \le 2^{\iota(d_1+d_2)}.
\]
Hence we may estimate the last term of \cref{eq:intro-hoelder-2} by
\[
R^{d/q} \norm{F^{(\iota)}}_2 \norm{f}_p
 \lesssim 2^{-\varepsilon\iota} \norm{F^{(\iota)}}_{L^2_{d/q+\varepsilon}(\R)} \norm{f}_p.
\]
The case $\ell>\iota$ can be treated similarly by using \cref{eq:intro-hoelder-1} in conjunction with \cref{eq:intro-restriction-2} and summing the geometric series over all $\ell>\iota$. As a consequence, only regularity $\delta>d/q-1/2$ of the multiplier $F(\lambda)=(1-\lambda^2)_+^\delta$ is necessary for the $L^p$-boundedness of the Bochner--Riesz mean $F(\sqrt L)=(1-L)_+^\delta$. Thus, at least by our heuristics, the truncation of \cref{eq:intro-truncation} ultimately provides thresholds with the topological dimension instead of the homogeneous one.

It should be emphasized that the present approach benefits from the fact that the dimension $d_2$ of the second layer $\g_2$ of the Lie algebra $\g$ is smaller than the dimension of the first layer. Doing the decomposition in the first layer reflects a phenomenon that has already been prominent in the setting of Grushin operators $-\Delta_x-|x|^2\Delta_u$, $(x,u)\in \R^{d_1}\times \R^{d_2}$, which are closely related to sub-Laplacians on Heisenberg type groups. In \cite{MaSi12}, Martini and Sikora proved a Mikhlin--Hörmander type result with threshold $s>D/2$ for Grushin operators, where $D:=\max\{d_1+d_2,2d_2\}$, which was later improved by Martini and Müller \cite{MaMue14a} to hold for the topological dimension $d$ in place of $D$. Both approaches rely on weighted Plancherel estimates for the associated integral kernels, but differ by the employed weights. While \cite{MaSi12} uses the weight $|x|^\gamma$ in the first layer in conjunction with a sub-elliptic estimate, \cite{MaMue14a} employs the weight $|u|^\gamma$ in the second layer. A similar phenomenon occurred later in the articles \cite{ChOu16} and \cite{Ni21}, where spectral multiplier theorems with $p$-specific regularity bounds were proved. The approach of Chen and Ouhabaz \cite{ChOu16} relies on weighted restriction type estimates using the weight $|x|^\gamma$ in the first layer, while \cite{Ni21} employs the weighted Plancherel estimate of \cite{MaMue14a} with weight in the center to express support conditions of integral kernels. In accordance with the phenomenon of \cite{MaSi12} and \cite{MaMue14a}, the result of \cite{ChOu16} needs $s>D\left(1/p-1/2\right)$ as regularity condition, while \cite{Ni21} only needs regularity of order $s>d\left(1/p-1/2\right)$. The present article relies on the same key idea as in \cite{Ni21}, but uses instead a weighted Plancherel estimate featuring the weight $|x|^\gamma$. However, this approach still provides the optimal threshold $s>d\left(1/p-1/2\right)$ in \cref{thm:main} and is in line with the phenomena described above, since $d_2< d_1$ and thus $D=d$, due to the fact that $G$ is a Heisenberg type group.

\subsection{Structure of the paper} \Cref{sec:spectral-theory} and \ref{sec:cc} are preliminary sections dealing with the spectral theory and sub-Riemannian geometry of sub-Laplacians. In \cref{sec:rest}, we prove the previously mentioned truncated restriction type estimates. \Cref{sec:weighted-plancherel} is devoted to proving a weighted Plancherel estimate with weight $|x|^\gamma$. In \cref{sec:abstract} we reduce the proof of \cref{thm:main} to spectral estimates for multipliers whose Fourier transform is supported on dyadic scales. The proof of this reduced version of \cref{thm:main} is given in \cref{sec:main-proof}. \Cref{sec:weighted} is an additional section, where we show that a sub-elliptic estimate for the sub-Laplacian with the help of which one could directly transfer the approach of \cite{ChOu16} to the setting of Heisenberg type groups is in general false.

\subsection{Notation}
We let $\N=\{0,1,2,\dots\}$. The space of (equivalence classes of) integrable simple functions on a two-step stratified Lie group $G$ will be denoted by $D(G)$, while $\S(G)$ shall denote the space of Schwartz functions on $G\cong\R^d$. The indicator function of a subset $A$ of some measurable space will be denoted by $\textbf{1}_A$. Given two suitable functions $f_1,f_2$ on $G$, let $f*g$ denote the group convolution given by
\[
f*g(x,u) = \int_{G} f(x',u')\,g\big((x',u')^{-1}(x,u)\big) \,d(x',u'),\quad (x,u)\in G,
\]
where $d(x',u')$ denotes the Lebesgue measure on $G$. For a function $f\in L^1(\R^n)$, the Fourier transform $\hat f$ is defined by
\[
\hat f(\xi) = \int_{\R^n} f(x) e^{-i\xi x}\, dx,\quad
\xi\in\R^n,
\]
while the inverse Fourier transform $\check f$ is given by
\[
\check f(x) = (2\pi)^{-n} \int_{\R^n} f(\xi) e^{ix\xi}\, d\xi,\quad
x\in\R^n.
\]
We write $A\lesssim B$ if $A\le C B$ for a constant $C$. If $A\lesssim B$ and $B\lesssim A$, we write $A\sim B$. Moreover, we fix the following dyadic decomposition throughout this article: Let $\chi:\R\to [0,1]$ be an even and smooth function such that $1/2\le |\lambda|\le 2$ for all $\lambda\in\supp\chi$ and
\[
\sum_{j\in\Z} \chi_j(\lambda) = 1 \quad \text{for } \lambda\neq 0,
\]
where $\chi_j$ is given by
\begin{equation}\label{eq:dyadic}
\chi_j(\lambda)=\chi(\lambda/2^j) \quad \text{for }j\in\Z.
\end{equation}

\subsection{Acknowledgments} I am deeply grateful to my advisor Detlef Müller for constant support and numberless helpful suggestions. I am also very grateful to Alessio Martini for bringing an error in \cref{lem:spec-twisted} of an earlier version of this article to my attention. I would also like to thank the anonymous referee for carefully reading the paper and making a number of helpful suggestions.

\section{Spectral theory of sub-Laplacians on Heisenberg type groups}
\label{sec:spectral-theory}

Let $G$ be a two-step stratified Lie group. Via exponential coordinates, we may identify $G$ with its Lie algebra $\g$, which is the tangent space at the identity of $G$. Since $G$ is stratified of step~2, $\g$ can be decomposed as $\g=\g_1\oplus \g_2$ with $[\g_1,\g_1]=\g_2$ and $\g_2\subseteq \g$ being contained in the center of $\g$. Let $\dot\g_2^*=\g_2^*\setminus\{0\}$. For any $\mu\in \g_2^*$, let $\omega_\mu$ be the skew-symmetric bilinear form given by
\[
\omega_\mu(x,x')=\mu([x,x']), \quad x,x'\in\g_1.
\]
Then $G$ is called a \textit{Heisenberg type group} if there is an inner product $\langle \cdot,\cdot\rangle$ on $\g$ with respect to which the decomposition $\g=\g_1\oplus \g_2$ is orthogonal, and the skew-symmetric endomorphisms $J_\mu$ given by $\omega_\mu(x,x') = \langle J_\mu x,x' \rangle$ for all $x,x'\in\g_1$ satisfy
\[
J_\mu^2 = - |\mu|^2\id_{\g_1} \quad\text{for all } \mu\in\g_2^*,
\]
where $|\cdot|$ is the norm on $\g_2^*$ induced by the inner product $\langle \cdot,\cdot \rangle$.

For the rest of this section, we assume that $G$ is a Heisenberg type group. In particular, this implies that
\begin{itemize}
    \item $\dim\g_1$ is even,
    \item $\dim\g_2 < \dim\g_1$ (since $\g_2^*\to(\g_1/\R x')^*,\mu\to \omega_\mu(\cdot,x')$ is injective for $x'\neq 0$).
\end{itemize}
Let $d_1 = \dim \g_1$ and $d_2 = \dim \g_2$. We fix an orthonormal basis $X_1,\dots,X_{d_1}$ of~$\g_1$ and an orthonormal basis $U_1,\dots,U_{d_2}$ of $\g_2$. In the following, to simplify our notation, we identify $G$ and $\g=\g_1\oplus \g_2$ with $\R^{d_1}\times \R^{d_2}$ via the chosen basis. Note that the group multiplication is then given by
\[
(x,u)(x',u')=(x+x',u+u'+\tfrac 1 2 [x,x']),\quad x,x'\in \g_1,u,u'\in \g_2.
\]
As usual, the tangent space $\g$ is in turn identified with the Lie algebra of (smooth) left-invariant vector fields on $G$ via the Lie-derivative. Given a smooth function $f$ on $G$, we have
\begin{align*}
X_j f(x,u)
& = \frac{d}{dt} f((x,u)(tX_j,0))\big|_{t=0} \\
& = \partial_{x_j} f(x,u) + \frac 1 2 \sum_{k=1}^{d_2} \langle U_k,[x,X_j]\rangle \partial_{u_k} f(x,u),\\
U_k f(x,u)
& = \partial_{u_k} f(x,u).
\end{align*}
The sub-Laplacian $L$ associated with the vector fields $X_1,\dots,X_{d_1}$ is the second order differential operator given by
\[
L = -(X_1^2+\dots + X_{d_1}^2).
\]
For $f\in L^1(G)$ and $\mu\in\g_2^*$, let $f^\mu$ denote the $\mu$-section of the partial Fourier transform along the second layer $\g_2$ given by
\begin{equation}\label{eq:partial-ft-2}
f^\mu(x) = \int_{\g_2} f(x,u) e^{- i \langle\mu, u\rangle} \, du,\quad x\in \g_1.
\end{equation}
Up to some constant, this defines an isometry $\mathcal F_2:L^2(\g_1\times \g_2)\to L^2(\g_1\times \g_2^*)$. Given $f\in L^2(G)$, we also write $f^\mu=(\mathcal F_2 f)(\cdot,\mu)$ (for almost all $\mu\in\g_2^*$) in the following. For fixed $\mu\in\g_2^*$, we have
\[
(X_j f)^\mu = X_j^\mu f^\mu,
\]
where $X_j^\mu$ is the differential operator on $\g_1$ given by
\[
X_j^\mu = \partial_{x_j} + \tfrac i 2 \mu([x,X_j]).
\]
Let $L^\mu$ be the \textit{$\mu$-twisted Laplacian} on $\g_1$ given by
\begin{equation}\label{eq:twisted-laplace}
L^\mu = -((X_1^\mu)^2+\dots+(X_{d_1}^\mu)^2).
\end{equation}
Let $n:=d_1/2$ and $S:=\{\mu\in \g_2^* : |\mu|=1\}$. Note that $J_\mu$ is orthogonal with eigenvalues $\pm i$ for $\mu\in S$. Hence, for any $\mu\in\dot\g_2^*$, there is an orthogonal matrix $T_\mu=T_{\bar\mu}$ on $\g_1=\R^{2n}$, where $\bar\mu:=\mu/|\mu|\in S$, such that
\[
\omega_\mu(T_{\bar\mu}z,T_{\bar\mu}w) = |\mu| \omega(z,w)\quad \text{for all } z,w \in \R^{2n},
\]
where $\omega(z,w)=(Jz)^\top w$ is the standard symplectic form\footnote{A prominent choice in the literature is also the symplectic form $(z,w)\mapsto\Im(z\bar w)=-\omega(z,w)$, which would correspond to the transpose of the matrix in \cref{eq:standard-symplectic}. The reason for choosing $\omega$ as above instead of $(z,w)\mapsto\Im(z\bar w)$ is that we want to have the same sign for the third summand in \cref{eq:classical-twisted} as in (1.3.14) of \cite{Th93}.} on $\R^{d_1}$ which is induced by the $d_1\times d_1$ matrix
\begin{equation}\label{eq:standard-symplectic}
J=
\begin{pmatrix}
0 & -\id_{\R^n} \\
\id_{\R^n} & 0
\end{pmatrix}.
\end{equation}
By compactness of $S$, we may assume that the map $\dot\g_2^*\ni\mu\mapsto T_{\bar\mu}$ is measurable, see for instance \cite[Thm.~1]{Az74}. Let $v^\mu_j$ be the $j$-th column of $T_{\bar\mu}^{-1}$ and $\nabla$ denote the usual gradient on $\R^{2n}$. Then, given a smooth function $g$ on $\R^{2n}$, we obtain
\begin{align*}
X_j^\mu (g\circ T_{\bar\mu}^{-1}) (T_{\bar\mu} z)
& = \nabla g(z)v^\mu_j + \tfrac i 2 \omega_\mu(T_{\bar\mu} z, X_j) g(z) \\
& = \nabla g(z)v^\mu_j + \tfrac i 2 |\mu| \omega(z,v^\mu_j) g(z).
\end{align*}
Since $\sum_{j=1}^{d_1}v^\mu_j(v^\mu_j)^\top$ is the identity matrix on $\R^{d_1}$ thanks to orthogonality, the $\mu$-twisted Laplacian $L^\mu$ of \cref{eq:twisted-laplace} transforms into
\begin{equation}\label{eq:rotation-twisted}
L^\mu (g\circ T_{\bar\mu}^{-1}) (T_{\bar\mu} z) = L_0^{|\mu|} g(z),
\end{equation}
where $L_0^\lambda$ is the \textit{$\lambda$-twisted Laplacian} on $\R^{2n}$ given by
\begin{equation}\label{eq:classical-twisted}
L_0^\lambda = -\Delta_z + \tfrac 1 4 \lambda^2 |z|^2 - i\lambda \sum_{j=1}^n (a_j \partial_{b_j} - b_j \partial_{a_j}),\quad\lambda > 0,
\end{equation}
where we write $z\in\R^{2n}$ as $z=(a_1,\dots,a_n,b_1,\dots,b_n)$.
The $\lambda$-twisted Laplacian $L_0^\lambda$ admits a complete orthonormal system of eigenfunctions, which are given by the matrix coefficients of the Schrödinger representation, see \cite[Section 1.3]{Th93}. More precisely, let $\pi_\lambda:\mathbb H_n\to \mathcal U (L^2(\R^n))$, where $\mathcal U (L^2(\R^n))$ is the group of unitary operators on $L^2(\R^n)$, denote the Schrödinger representation of the Heisenberg group $\mathbb H_n=\R^{2n}\times \R$ on $L^2(\R^n)$ given by
\[
\pi_\lambda(a, b, t) \varphi(\xi) = e^{i \lambda t} e^{i \lambda(a \xi+\frac{1}{2} ab)} \varphi(\xi+b),
\]
where $a,b\in\R^n,t\in\R$ and $\varphi\in L^2(\R^n),\xi\in\R^n$. Moreover, let $\Phi_\nu^{\lambda}$ be the Hermite function defined by
\[
\Phi_\nu^\lambda(\xi) = \lambda^{n/4} \prod_{j=1}^n h_{\nu_j}(\lambda^{1/2} \xi_j),\quad \xi\in \R^n,
\]
where $h_\ell$ shall denote the $\ell$-th Hermite function on $\R$ given by
\[
h_\ell(t) = (-1)^\ell (2^\ell \ell! \sqrt{\pi})^{-1/2} e^{t^2/2} \Big(\frac{d}{dt}\Big)^\ell (e^{-t^2}),\quad t\in \R.
\]
Then, by Theorems 1.3.2 and 1.3.3 of \cite{Th93}, the matrix coefficients $\Phi_{\nu,\nu'}^\lambda$, $\nu,\nu'\in\N^n$ given by
\begin{equation}\label{eq:special-hermite}
\Phi_{\nu,\nu'}^\lambda(z) := (2\pi)^{-n/2} \lambda^{n/2} (\pi_\lambda(z,0) \Phi_\nu^\lambda, \Phi_{\nu'}^\lambda),\quad z\in\R^{2n}
\end{equation}
form a complete orthonormal system of eigenfunctions of $L_0^\lambda$, with 
\begin{equation}\label{eq:eigenfunction-1}
L_0^\lambda\Phi_{\nu,\nu'}^\lambda = (2|\nu'|_1+n) \lambda \Phi_{\nu,\nu'}^\lambda,
\end{equation}
where $(\cdot,\cdot)$ is the inner product on $L^2(\R^{2n})$, and $|\nu|_1=\nu_1+\dots+\nu_n$ denotes the length of the multiindex $\nu\in\N^n$. Hence $L^2(\R^{2n})$ decomposes into eigenspaces of $L_0^\lambda$, where the orthogonal projection $\Lambda_k^\lambda$ onto the eigenspace of the eigenvalue $(2k+n)\lambda$, $k\in\N$ is given by
\[
\Lambda_k^\lambda g = \sum_{\nu\in \N^n} \sum_{|\nu'|_1 =k} (g,\Phi_{\nu,\nu'}^\lambda)\Phi_{\nu,\nu'}^\lambda,\quad g \in L^2(\R^{2n}).
\]
Via the transformation $T_{\bar\mu}$, the spectral decomposition of the $\mu$-twisted Laplacian $L^\mu$ of \cref{eq:twisted-laplace} can be expressed in terms of the spectral decomposition of $L_0^\lambda$, which follows directly from \cref{eq:rotation-twisted} and \cref{eq:eigenfunction-1}. In the following, we put
\[
[k]:=2k+n, \quad k\in\N.
\]

\begin{lemma}\label{lem:spec-twisted}
For $\mu\in \dot \g_2^*$, the operator $L^\mu$ on $L^2(\g_1)$ admits an orthonormal basis of eigenfunctions associated with the eigenvalues $[k]|\mu|$, $k\in\N$. The orthogonal projection $\Pi_k^\mu$ onto the eigenspace of the eigenvalue $[k]|\mu|$ is given by
\begin{equation}\label{eq:func-calc-proj}
\Pi_k^\mu g = (\Lambda_k^{|\mu|} (g\circ T_{\bar\mu})) \circ T_{\bar\mu}^{-1}, \quad g\in L^2(\g_1).
\end{equation}
\end{lemma}

\begin{remark}
\Cref{lem:spec-twisted} is essentially Proposition 4.5 of \cite{CaCi13}. Note, however, that the definitions of $\Lambda_k^\lambda$ and $\Pi_k^\mu$ in \cite{CaCi13} differ from ours by the factors $\lambda^n$ and $|\mu|^n$. Moreover, \cite{CaCi13} even states that the above result holds for left-invariant sub-Laplacians on the larger class of Métivier groups. Unfortunately, this is not the case in general, since the eigenvalues of the twisted Laplacian $L^\mu$ are not necessarily of the form $[k]|\mu|$ if $G$ is only assumed to be a Métivier group. An example may be found for instance in \cite[Eq.\ (2.3), (2.4)]{MueSt94}, where the corresponding eigenvalues are of the form
\[
\sum_{j=1}^n a_j(2\nu_j+1)|\mu|,\quad  \nu\in\N^n.
\]
See also \cite[Sec.\ 2]{MaMue14b} for a further discussion. However, the results of \cite{CaCi13} remain true under the additional hypothesis that $G$ is a Heisenberg type group, and the restriction theorem of \cite{CaCi13} can be expected to hold in greater generality.

A restriction type estimate that holds beyond Heisenberg type groups can be found in the follow-up paper \cite{Nie23b}, but unfortunately the estimate there does not seem to be sufficient to recover the result claimed in \cite{CaCi13}.
\end{remark}

The projection $\Lambda_k^\lambda$ can be written in a more explicit form as a twisted convolution with a Laguerre function. For $\lambda>0$, let $f\times_\lambda g$ be the $\lambda$-twisted convolution given by
\[
f\times_\lambda g (z) = \int_{\R^{2n}} f(w)g(z-w) e^{\frac i 2 \lambda \omega(z,w)}\,dw,\quad z\in\R^{2n},
\]
where $\omega$ is again the standard symplectic form induced by the matrix $J$ in \cref{eq:standard-symplectic}. Moreover, let $\varphi_k^\lambda$ be the Laguerre function given by
\begin{equation}\label{eq:laguerre-function-1}
\varphi_k^\lambda(z) = \lambda^n L_k^{n-1}(\tfrac 1 2 \lambda |z|^2) e^{-\frac 1 4 \lambda |z|^2},\quad z\in\R^{2n},
\end{equation}
where $L_k^{n-1}$ denotes the $k$-th Laguerre polynomial of type $n-1$. Then, since 
\begin{equation}\label{eq:eigenfunction-3}
\Phi_{\nu,\nu'}^\lambda(z) = \lambda^{n/2} \Phi^1_{\nu,\nu'}(\lambda^{1/2}z)
\end{equation}
by the definition \cref{eq:special-hermite} of $\Phi_{\nu,\nu'}^\lambda$, (1.3.41) and (1.3.42) of \cite[pp.\ 21]{Th93} imply
\begin{equation}\label{eq:laguerre-function-2}
\varphi_k^\lambda(z) = (2\pi)^{n/2} \lambda^{n/2} \sum_{|\nu|_1=k} \Phi^\lambda_{\nu,\nu}(z).
\end{equation}
Hence, by (2.1.5) of \cite[p.\ 30]{Th93}, $\Lambda_k^\lambda$ may be rewritten as
\begin{equation}\label{eq:proj-compact}
\Lambda_k^\lambda g = g \times_\lambda \varphi_k^\lambda.
\end{equation}

\begin{remark}
Our definition of $\varphi_k^\lambda$ differs from that of \cite{CaCi13} by the factor $|\lambda|^{n/2}$. 
\end{remark}

The operators $L,-iU_1,\dots,-iU_{d_2}$ (where $U_1,\dots,U_{d_2}$ is the chosen basis of the second layer $\g_2$) form a system of formally self-adjoint, left-invariant and pairwise commuting differential operators, whence they admit a joint functional calculus \cite{Ma11}. It is well-known \cite[Section 1]{MueRiSt96} that for suitable functions $F:\R\times \R^{d_2}\to\C$, the operator $F(L,\mathbf U)$ with $\mathbf U:=(-iU_1,\dots,-iU_{d_2})$ possesses a convolution kernel that can be expressed in terms of the Fourier transform and Laguerre functions $\varphi_k^{\lambda}$. We provide a direct argument here, although alternatively, the convolution kernel can also be computed by using the Fourier inversion formula of the group Fourier transform on $G$ and the fact that the unitary group representations and the joint functional calculus of $L,-iU_1,\dots,-iU_{d_2}$ are compatible, see Proposition~1.1 and Lemma~2.2 of \cite{Mue90}.

\begin{proposition}\label{prop:conv-kernel}
Let $F:\R\times\R^{d_2}\to\C$ be a bounded Borel function. Then
\begin{equation}\label{eq:func-calc}
(F(L,\mathbf U)f)^\mu(x) = F(L^\mu,\mu) f^\mu(x)
\end{equation}
for all $f\in L^2(G)$ and almost all $x\in\g_1$, $\mu\in\g_2^*$.
If $F$ is additionally compactly supported in $\R\times(\R^{d_2}\setminus\{0\})$, then $F(L,\mathbf U)$ possesses a convolution kernel $\mathcal K_{F(L,U)}$, i.e.,
\[
F(L,\mathbf U) f = f * \mathcal K_{F(L,\mathbf U)} \quad \text{for all } f\in \S(G),
\]
which is given by
\begin{align}\label{eq:conv-kernel}
\mathcal K_{F(L,\mathbf U)}(x,u) = (2\pi)^{-d_2} \sum_{k=0}^\infty \int_{\dot\g_2^*} & F([k]|\mu|,\mu) \varphi_k^{|\mu|}(x) e^{i\langle \mu, u\rangle} \, d\mu
\end{align}
for almost all $(x,u)\in G$.
\end{proposition}

\begin{proof}
The identity \cref{eq:func-calc} can be proved by the same approach as in the proof of Proposition~5 of \cite{MaSi12} by writing down the corresponding functional calculi in terms of the Fourier transform and the orthogonal projections provided by the eigenfunctions of the $\mu$-twisted Laplacian $L^\mu$. To prove \cref{eq:conv-kernel}, we observe that \cref{eq:func-calc} and \cref{lem:spec-twisted} yield
\begin{equation}\label{eq:proof-conv-kernel-1}
F(L,\mathbf U)f(x,u) = (2\pi)^{-d_2} \sum_{k=0}^\infty \int_{\dot\g_2^*} F([k]|\mu|,\mu) \Pi_k^\mu f^\mu(x) e^{i\langle \mu, u\rangle} \,d\mu.
\end{equation}
By \cref{lem:spec-twisted}, \cref{eq:proj-compact}, and the fact that $\varphi_k^{|\mu|}$ is radial-symmetric,
\begin{align}
\Pi_k^\mu f^\mu(x)
 & = (\Lambda_k^{|\mu|} (f^\mu\circ T_{\bar\mu})) (T_{\bar\mu}^{-1} x) \notag \\
 & = \big((f^\mu\circ T_{\bar\mu}) \times_{|\mu|} \varphi_k^{|\mu|} \big) (T_{\bar\mu}^{-1} x) \notag \\
 & = \int_{\R^{2n}} f^\mu(T_{\bar\mu}w)\varphi^{|\mu|}_k(T_{\bar\mu}^{-1} x-w) e^{\frac i 2 |\mu| \omega( T_{\bar\mu}^{-1} x, w)} \,dw \notag \\
 & = \int_{\g_1} f^\mu(x')\varphi^{|\mu|}_k(x-x') e^{\frac i 2 \omega_\mu(x,x')} \,dx', \label{eq:proof-conv-kernel-2} 
\end{align}
where $\omega$ denotes again the standard symplectic form on $\R^{2n}$ associated with the matrix $J$ from \cref{eq:standard-symplectic}. Plugging \cref{eq:proof-conv-kernel-2} into \cref{eq:proof-conv-kernel-1}, unboxing the Fourier transform $f^\mu$, and rearranging the order of integration yields
\begin{align*}
F(L,\mathbf U)f(x,u)
= (2\pi)^{-d_2} \int_G  f(x',u') \sum_{k=0}^\infty \int_{\dot\g_2^*} & F([k]|\mu|,\mu) \varphi^{|\mu|}_k(x-x') \\ 
& e^{i\langle \mu, u-u'\rangle} e^{\frac i 2 \omega_\mu(x,x')} \,d\mu \,d(x',u').
\end{align*}
By definition, we have
\[
f * \mathcal K_{F(L,\mathbf U)}(x,u) = \int_G  f(x',u') \mathcal K_{F(L,\mathbf U)}(x-x',u-u'-\tfrac 1 2 [x',x]) \,d(x',u').
\]
This yields \cref{eq:conv-kernel}.
\end{proof}

\section{Truncated restriction type estimates}
\label{sec:rest}

In this section, we prove the truncated restriction type estimates for the sub-Laplacian $L=(X_1^2+\dots+X_{d_1}^2)$, given that $G$ is a Heisenberg type group and $X_1,\dots,X_{d_1}$ is an orthonormal basis on the first layer $\g_1$ of the stratification $\g=\g_1\oplus\g_2$. As in \cite{CaCi13} (and similarly in \cite{ChOu16,Ni21}), the idea of the proof is to first apply a restriction type estimate in the variable $x\in \g_1$ for the $\mu$-twisted Laplacian $L^\mu$ given by \cref{eq:twisted-laplace} and then the Stein--Tomas restriction estimate in the central variable $u\in \g_2$. The restriction type estimate for the orthogonal projection $\Pi_k^\mu$ onto the $k$-th eigenspace of the $\mu$-twisted Laplacian $L^\mu$ is given by the following lemma, which is Lemma 4.7 of \cite{CaCi13}. Let again $[k]=2k+d_1/2$ and $n=d_1/2$.

\begin{lemma}\label{lem:discrete-restriction}
If $1\le p \le 2(d_1 +1)/(d_1+3)$, then
\begin{equation}\label{eq:discrete-restriction}
\norm{\Pi_k^\mu}_{L^p(\g_1)\to L^2(\g_1)} \le C_p |\mu|^{n (\frac 1 p - \frac 1 2)} [k]^{n (\frac 1 p - \frac 1 2)-\frac 1 2} \quad\text{for all } k\in \N.
\end{equation}
\end{lemma}

\begin{proof}
By Theorem 1 of \cite{KoRi07}, we have $\norm{\Lambda_k^1}_{2\to p'} \lesssim [k]^{n ( \frac 1 2 - \frac 1  {p'} )- \frac 1 2}$, so by duality
\[
\norm{\Lambda_k^1}_{p\to 2}\lesssim [k]^{n ( \frac 1  p -\frac 1 2 )- \frac 1 2}.
\]
In view of \cref{eq:eigenfunction-3}, rescaling with $\lambda^{1/2}$ yields
\[
\norm{\Lambda_k^\lambda}_{p\to 2}\lesssim \lambda^{n (\frac 1 p - \frac 1 2)} [k]^{n ( \frac 1 p - \frac 1 2)-\frac 1 2}\quad \text{for } \lambda > 0.
\]
Hence, together with \cref{eq:func-calc-proj} and a substitution, we obtain \cref{eq:discrete-restriction}.
\end{proof}

\begin{remark}
The condition $1\le p\le 2(d_2+1)/(d_2+3)$ of \cref{thm:main} (and \cref{prop:restriction} below) implies in particular $1\le p\le 2(d_1+1)/(d_1+3)$ since $d_1> d_2$ due to the fact that $G$ is a Heisenberg type group.
\end{remark}

Choosing a basis $U_1,\dots,U_{d_2}$ of the second layer $\g_2$, we define the operator
\[
U:=(-(U_1^2+\dots+U_{d_2}^2))^{1/2}.
\]
Now we state the truncated restriction type estimates for the sub-Laplacian $L$.

\begin{theorem}[Truncated restriction type estimates]\label{prop:restriction}
Suppose that $1\le p\le 2(d_2+1)/(d_2+3)$. Let $F:\R\to\C$ be a bounded Borel function supported in $[1/8,8]$, and, for $\ell\in\N$, let $F_\ell : \R\times \R \to \C$ be given by
\[
F_\ell(\lambda,\rho) = F(\sqrt \lambda) \chi_\ell(\lambda/\rho)\quad\text{for } \lambda\ge 0,\rho\neq 0,
\]
and $F_\ell(\lambda,\rho)=0$ else, where $(\chi_\ell)_{\ell \in\Z}$ is the dyadic decomposition of \cref{eq:dyadic}. Then
\begin{equation}\label{eq:restriction}
\norm{ F_\ell(L,U) }_{p\to 2}
\le C_{p} 2^{-\ell d_2(\frac 1p - \frac 1 2)} \norm{F}_2 \quad \text{for all }\ell\in\N.
\end{equation}
\end{theorem}

\begin{remark}
Note that $d_1$ is even since $G$ is a Heisenberg type group. Thus, $\chi_\ell([k])=0$ for all $k\in\N$ whenever $\ell<0$. Hence \cref{eq:proof-conv-kernel-1} yields
\[
\sum_{\ell=0}^\infty F_\ell(L,U) f = F(\sqrt L)f.
\]
\end{remark}

\begin{proof}
Let $f\in \S(G)$. Given $\mu\in\g_2^*$ and $k\in \N$, we write $g_k^\mu = F( \sqrt{[k]|\mu|}) f^\mu$, where $f^\mu$ denotes again the partial Fourier transform in $\mu$. Note that $[k]\sim 2^\ell$ for $([k]|\mu|,|\mu|)\in \supp F_\ell$. Using Plancherel's theorem, \cref{eq:func-calc}, and orthogonality in $L^2(\g_1)$, we obtain
\begin{align}
\norm{ F_\ell(L,U) f }^2_{L^2(G)}
& \sim \int_{\g_2^*} \int_{\g_1} |F_\ell(L^\mu,|\mu|) f^\mu(x)|^2 \,dx \, d\mu \notag \\
& = \int_{\dot\g_2^*} \int_{\g_1} \bigg|\sum_{k=0}^\infty F(\sqrt{[k]|\mu|}) \chi_\ell ([k]) \Pi_k^\mu f^\mu(x) \bigg|^2 \,dx \, d\mu \notag \\
& \lesssim \sum_{[k]\sim 2^\ell} \int_{\dot\g_2^*} \norm{ \Pi_k^\mu g_k^\mu }^2_{L^2(\g_1)}\, d\mu. \label{eq:proof-rest-1}
\end{align}
Now \cref{lem:discrete-restriction} yields
\begin{align}
\norm{ \Pi_k^\mu g_k^\mu }_{L^2(\g_1)}^2
& \lesssim |\mu|^{d_1( \frac 1 p - \frac 1 2)} [k]^{d_1( \frac 1 p - \frac 1 2)-1} \norm{ g_k^\mu }_{L^p(\g_1)}^2 \notag \\
& \sim [k]^{-1} \norm{ g_k^\mu}_{L^p(\g_1)}^2 .
\label{eq:proof-rest-2}
\end{align}
In the last line we used the fact that $[k] |\mu|\sim 1$ whenever $[k] |\mu|\in \supp F$. Moreover, since $2/p\ge 1$, Minkowski's integral inequality yields
\begin{equation}\label{eq:proof-rest-3}
\int_{\dot\g_2^*} \norm{g_k^\mu}_{L^p(\g_1)}^2\,d\mu
 \le \bigg(\int_{\g_1}\bigg( \int_{\dot\g_2^*} |g_k^\mu(x)|^2 \,d\mu\bigg)^{\frac p 2}\,dx \bigg)^{\frac 2 p}.
\end{equation}
Let $f_{x}:=f(x,\cdot)$ and $\widehat\cdot$ denote the Fourier transform on $\g_2$. Using polar coordinates and applying the Stein--Tomas restriction estimate \cite{To79} yields
\begin{align}
\int_{\dot\g_2^*} |g_k^\mu(x)|^2 \,d\mu
 & = \int_0^\infty \int_{S^{d_2-1}} | F(\sqrt{[k]r}) \widehat{f_{x}}(r\omega)|^2 r^{d_2-1} \, d\sigma(\omega) \, dr\notag\\
 & = \int_0^\infty | F(\sqrt{[k]r})|^2 r^{-d_2-1}\int_{S^{d_2-1}} \big| \big(f_{x}(r^{-1}\,\cdot\,)\big)^\wedge(\omega)\big|^2 \, d\sigma(\omega) \, dr \notag \\
 & \lesssim \int_0^\infty | F(\sqrt{[k]r})|^2  r^{-d_2-1} \norm{f_{x}(r^{-1}\,\cdot\,)}_{L^p(\g_2)}^2 \, dr\notag \\
 & = \int_0^\infty |F(\sqrt{[k]r})|^2 r^{2d_2(\frac 1 p -\frac 1 2)-1}  \, dr \, \norm{f_{x}}_{L^p(\g_2)}^2\notag \\
 & \sim [k]^{-2d_2(\frac 1 p -\frac 1 2)} \norm{F}_{L^2(\R)}^2 \norm{f_{x}}_{L^p(\g_2)}^2.\notag 
\end{align}
In combination with \cref{eq:proof-rest-1}, \cref{eq:proof-rest-2} and \cref{eq:proof-rest-3}, we obtain
\begin{align*}
\norm{ F_\ell(L,U) f }^2_{L^2(G)}
& \lesssim \sum_{[k]\sim 2^\ell} [k]^{-2d_2(\frac 1 p - \frac 1 2)-1} \norm{F}_2^2 \norm{f}_p^2 \\
& \sim 2^{-2\ell d_2(\frac 1 p - \frac 1 2)} \norm{F}_2^2 \norm{f}_p^2.
\end{align*}
This proves \cref{eq:restriction}.
\end{proof}

\section{A weighted Plancherel estimate}\label{sec:weighted-plancherel}

In this section, we prove a weighted Plancherel estimate for convolution kernels associated with the sub-Laplacian~$L$ on the Heisenberg type group $G$. Usually, those estimates are the crux of the matter when proving Mikhlin--Hörmander results featuring the threshold $s>d/2$, where $d$ is the topological dimension of the underlying space, see for example \cite[Thm.~4.6]{Ma12} or \cite[Prop.~3]{Ma15}. However, in the present setting, the weighted Plancherel estimate \cref{eq:weighted-plancherel} will serve a different purpose, namely turning support conditions in conjunction with convolution kernels into some sort of rapid decay.

\begin{proposition}\label{prop:weighted-plancherel}
Let $F$ and $F_\ell$ be defined as in \cref{prop:restriction}, and $\mathcal K_\ell$ be the convolution kernel of the operator $F_\ell(L,U)$. Then, for all $\alpha\ge 0$,
\begin{equation}\label{eq:weighted-plancherel}
\int_G \big| |x|^\alpha \mathcal K_\ell(x,u) \big|^2 \,d(x,u) \le C_{\alpha} 2^{\ell(2\alpha - d_2)} \norm{F}_{L^2}^2\quad\text{for all } \ell\in\N.
\end{equation}
\end{proposition}

\begin{proof}
Let $\alpha\ge 0$. Using \cref{eq:conv-kernel} in combination with Plancherel's theorem, we obtain
\begin{align}
 \int_G \big||x|^\alpha & \mathcal K_\ell(x,u) \big|^2 \,d(x,u) \notag \\
& \sim \int_{\dot\g_2^*}\int_{\g_1} \Big| |x|^\alpha \sum_{k=0}^\infty F_\ell\big([k]|\mu|,|\mu|\big) \varphi_k^{|\mu|}(x)\Big|^2\, dx \, d\mu.\label{eq:plancherel-5}
\end{align}
Given $\mu\in\dot \g_2^*$, we consider the rescaled Hermite operator $H^\mu =- \Delta_z + \tfrac 1 4 |z|^2 |\mu|^2$ acting on $L^2(\R^{2n})$. By Proposition 3.3 of \cite{ChOu16},
\begin{equation}\label{eq:sub-elliptic}
\norm*{|\cdot|^\alpha g}_{L^2(\R^{2n})} \lesssim |\mu|^{-\alpha} \norm{(H^\mu)^{\alpha/2} g}_{L^2(\R^{2n})}.
\end{equation}
On the other hand, by Equation (1.3.25) of \cite{Th93}, the functions $\Phi^{|\mu|}_{\nu,\nu'}$ defined by \cref{eq:special-hermite} are also eigenfunctions of $H^\mu$, with
\[
H^\mu \Phi^{|\mu|}_{\nu,\nu'} = (|\nu|+|\nu'|+n)|\mu| \Phi^{|\mu|}_{\nu,\nu'} \quad\text{for all } \nu,\nu'\in\N^n.
\]
By the definition \cref{eq:laguerre-function-2} of $\varphi_k^{|\mu|}$, this implies in particular
\[
H^\mu \varphi_k^{|\mu|} = [k] |\mu| \varphi_k^{|\mu|}.
\]
(Alternatively, one could use \cref{eq:eigenfunction-1} by exploiting that $\varphi_k^{|\mu|}$ is radial-symmetric by \cref{eq:laguerre-function-1} and that the operators $H^\mu$ and $L_0^{|\mu|}$ coincide on such functions.) Hence, together with \cref{eq:sub-elliptic}, the right-hand side of \cref{eq:plancherel-5} can be dominated by a constant times
\begin{equation}\label{eq:plancherel-4}
\int_{\dot\g_2^*} \int_{\R^{2n}} \Big|  \sum_{k=0}^\infty [k]^{\alpha/2} |\mu|^{-\alpha/2} F_\ell\big([k]|\mu|,|\mu|\big)  \varphi_k^{|\mu|}(x)\Big|^2\, dx \, d\mu.
\end{equation}
Using $[k]|\mu|\sim 1$ for $\sqrt{[k]|\mu|}\in \supp F$, $|F_\ell(\lambda,\rho)|\le |F(\lambda)|$ and orthogonality of the functions $\varphi_k^{|\mu|}$, \cref{eq:plancherel-4} can be estimated by a constant times
\begin{equation}\label{eq:plancherel-2}
\sum_{[k]\sim 2^\ell} [k]^{2\alpha} \int_{\dot\g_2^*} \int_{\R^{2n}} \big|F\big(\sqrt{[k]|\mu|}\big) \varphi_k^{\mu}(x)\big|^2\, dx \, d\mu.
\end{equation}
Since the functions $\Phi_{\nu,\nu'}^{|\mu|}$ form an orthonormal basis of $L^2(\R^{2n})$,
\[
|\mu|^{-n} \norm{\varphi_k^{|\mu|}}_2^2
= |\{\nu\in \N^{n}:|\nu|_1=k\}|
= \binom{k+n-1}{k} \sim (k+1)^{n-1}.
\]
Hence $\norm{\varphi_k^{|\mu|}}_2^2\sim [k]^{-1}$ for $[k]|\mu|\in \supp F$. Thus \cref{eq:plancherel-2} is comparable to
\begin{equation}\label{eq:plancherel-3}
\sum_{[k]\sim 2^\ell} [k]^{2\alpha-1} \int_{\dot\g_2^*} \big|F\big(\sqrt{[k]|\mu|}\big) \big|^2 \, d\mu.
\end{equation}
Using polar coordinates and a substitution in the integral over $\dot\g_2^*$ shows that \cref{eq:plancherel-3} in turn is comparable to
\[
\sum_{[k]\sim 2^\ell} [k]^{2\alpha-1-d_2} \norm{F}_{L^2(\R)}^2
 \sim 2^{\ell(2\alpha-d_2)} \norm{F}_{L^2(\R)}^2.
\]
This proves \cref{eq:weighted-plancherel}.
\end{proof}

\section{The sub-Riemannian geometry of the sub-Laplacian}\label{sec:cc}

In this section we summarize the main properties of the sub-Riemannian geometry associated with left-invariant sub-Laplacians on two-step stratified groups. Let $G$ be a two-step stratified Lie group and $\g=\g_1\oplus \g_2$ be a stratification of its Lie algebra. Let $X_1,\dots,X_{d_1}$ be a basis of $\g_1$ and $L=-(X_1^2+\dots+X_{d_1}^2)$ be the associated sub-Laplacian. We again identify $G\cong \g$ via the exponential map and $\g\cong \R^d$ by means of the basis $X_1,\dots,X_{d_1}$ of $\g_1$ and a basis $U_1,\dots,U_{d_2}$ of $\g_2$.

Let $d_{\mathrm{CC}}$ denote the Carnot--Carathéodory distance associated with the vector fields $X_1,\dots,X_{d_1}$. By definition, this means that for $g,h\in G$, the distance $d_{\mathrm{CC}}(g,h)$ is given by the infimum over all lengths of horizontal curves $\gamma:[0,1]\to G$ joining $g$ with $h$, see for instance \cite[Section III.4]{VaSaCo92}. Since the vector fields $X_1,\dots,X_{d_1}$ are left-invariant and $[\g_1,\g_1]=\g_2$, they satisfy Hörmander's condition \cite{Hoe67}, that is, the vector fields $X_1,\dots,X_{d_1}$ along with their iterated commutators
\[
[X_i,X_j],[X_i,[X_j,X_l]],\dots
\]
span the tangent space $\g=T_e G$ of $G$ at the identity $e\in G$, and hence at every point $g\in G$. (In our two-step setting, the vector fields $X_1,\dots,X_{d_1}$ together with their commutators $[X_i,X_j]$ already span the tangent space.) Hence, due to the Chow--Rashevskii theorem \cite[Proposition~III.4.1]{VaSaCo92}, $d_{\mathrm{CC}}$ is indeed a metric on $M$, which induces the (Euclidean) topology of $G=\R^d$.

Since $X_1,\dots,X_{d_1}$ are left-invariant vector fields, $d_{\mathrm{CC}}$ is left-invariant, that is,
\begin{equation}\label{eq:cc-translation}
d_{\mathrm{CC}}(ag,ah) = d_{\mathrm{CC}}(g,h)\quad\text{for all } a,g,h\in G.  
\end{equation}
On the other hand,
\[
\norm{ (x,u) } := (|x|^4+|u|^2)^{1/4},\quad (x,u)\in G
\]
defines a homogeneous norm in the sense of Folland and Stein \cite{FoSt82} with respect to the dilations $\delta_R$ given by
\begin{equation}\label{eq:cc-hom}
\delta_R(x,u)=(Rx,R^2u),\quad R\ge 0.
\end{equation}
Hence $G\times G\ni(g,h) \mapsto \norm{g^{-1}h}$ is a left-invariant (quasi-)distance on $G$. Since any two homogeneous norms on a homogeneous Lie group are equivalent \cite[Lemma~1.4]{FoSt82}, we have
\begin{equation}\label{eq:cc-equivalence}
d_{\mathrm{CC}}(g,h) \sim \norm{g^{-1}h} \quad \text{for all } g,h\in G. 
\end{equation}
Let $B_R^{d_{\mathrm{CC}}}(g)$ denote the ball of radius $R\ge 0$ centered at $g\in G$ with respect to $d_{\mathrm{CC}}$. Then \cref{eq:cc-translation} and \cref{eq:cc-hom} yield
\begin{equation}\label{eq:cc-volume}
|B_R^{d_{\mathrm{CC}}}(g)| = R^{Q} |B_1^{d_{\mathrm{CC}}}(0)|,
\end{equation}
where we identify $0$ with the identity element $e\in G$ via $G\cong \g$, and $Q=d_1+2d_2$ is the homogeneous dimension. Note that \cref{eq:cc-volume} yields in particular that the metric space $(G,d_{\mathrm{CC}})$ equipped with the the Lebesgue measure (which is a bi-invariant Haar measure on $G$) is a space of homogeneous type with homogeneous dimension~$Q$.

Furthermore, the sub-Laplacian $L$ possesses the finite propagation speed property with respect to the Carnot-Carathéodory distance $d_{\mathrm{CC}}$, which will be of fundamental importance in the proof of \cref{thm:main}.

\begin{lemma}\label{lem:finite-prop}
If $f,g\in L^2(G)$ are supported in open subsets $U,V\subseteq G$, then
\[
(\cos(t\sqrt L)f,g) = 0\quad\text{for all } |t|< d_{\mathrm{CC}}(U,V).
\]
\end{lemma}

For a proof, see \cite{Me84} or \cite[Corollary 6.3]{Mue04}.

\section{Reduction of \texorpdfstring{\cref{thm:main}}{Theorem 1.1} to dyadic spectral multipliers}\label{sec:abstract}

To prove \cref{thm:main}, we use the following general spectral multiplier result of \cite{ChOuSiYa16}, which allows us to reduce the spectral multiplier estimates of \cref{thm:main} to estimates for spectral multipliers whose Fourier transforms are supported on dyadic scales. Very similar arguments are used in \cite[Section~5]{ChOu16} and \cite[Section~4]{Ni21}, but we give a detailed discussion for the convenience of the reader. Given a suitable multiplier $F:\R\to\C$, we use the notation
\begin{equation}\label{eq:dyadic-piece}
F^{(\iota)} := (\hat F \chi_\iota)^\vee\quad \text{for }\iota\in\Z,
\end{equation}
where $\widehat\cdot$ and $\cdot^\vee$ denote the Fourier transform and its inverse on $\R$, respectively, and $(\chi_\iota)_{\iota \in\Z}$ is the dyadic decomposition from \cref{eq:dyadic}.

\begin{proposition}\cite[Proposition I.22]{ChOuSiYa16}\label{prop:COSY-original}
Let $(X,\rho,\mu)$ be a metric measure space of homogeneous type and $Q\ge 0$ such that
\[
\mu(B_{\lambda r}(x)) \le C \lambda^Q \mu(B_{r}(x)) \quad \text{for all } x\in X, r>0,\lambda\ge 1,
\]
where $B_s(x)$ denotes the ball of radius $s>0$ centered at $x\in X$.
Let $1\le p_0 < p <2$. Suppose that $L$ is a positive self-adjoint operator on $L^2(X)$ such that the following statements are satisfied:
\begin{enumerate}
\item[(i)] $L$ possesses the finite propagation speed property, that is,
\[
(\cos(t\sqrt L)f,g)_{L^2(X)} = 0\quad\text{for all } |t|< \rho(U,V)
\]
whenever $f,g\in L^2(X)$ are supported in open subsets $U,V\subseteq X$, where
\[
\rho(U,V) := \inf\{\rho(u,v):u\in U,v\in V \}.
\]
\item[(ii)] $L$ satisfies the \textit{Stein--Tomas restriction type condition} $(\mathrm{ST}^\infty_{p_0, 2})$ of \cite{ChOuSiYa16}, that is, for any $R>0$ and all bounded Borel functions $F:\R\to\C$ supported in $[0,R]$,
\begin{equation}\label{eq:ST-condition}
\|F(\sqrt L)(\mathbf{1}_{B_r(x)}f) \|_2 \leq C \mu(B_r(x))^{{\frac 1 2}-{\frac 1 {p_0}}} ( Rr )^{Q({\frac 1 {p_0}}-{\frac 1 2})}\|F\|_\infty \|f\|_{p_0}
\end{equation}
for all $x\in X$, all $r\geq 1/R$, and all $f\in L^{p_0}(X)$.
\item[(iii)] There is some $\beta > Q/2$ such that
\[
\sup_{t>0}\|F(t\sqrt L)\|_{p\to p}\leq C\|F\|_{L^\infty_\beta}
\]
for all even bounded Borel functions $F:\R\to\C$ with $\supp F \subseteq [-1, 1]$.
\end{enumerate}
Suppose that $F:\R\to\C$ is an even bounded Borel function and that there is a bounded sequence $(\alpha(\iota))_{\iota \in\Z}$ with $\sum_{\iota \ge 0}(\iota+1) \alpha(\iota)<\infty$ such that
\begin{equation}\label{eq:cond-1}
\norm{(F\chi_i)^{(j)}(\sqrt L)}_{p\to p} \le \alpha(i+j) \quad\text{for all } i,j\in \Z.
\end{equation}
Then the operator $F(\sqrt L)$ is of weak-type $(p,p)$.
\end{proposition}

\begin{remark}
To be precise, in \cite{ChOuSiYa16}, Proposition~I.22 requires the condition $(\mathrm E_{p_0,2})$ in place of the Stein--Tomas type restriction condition $(\mathrm{ST}^\infty_{p_0, 2})$. However, both conditions are equivalent by Proposition~I.3 of the same paper. Moreover, the notation of our decomposition indexed by $i$ and $j$ differs slightly from that of \cite{ChOuSiYa16} since $(F\eta_i)^{(j)}=(F\chi_{-i})^{(j)}$, where $(\eta_i)_{i\in\Z}$ is the dyadic decomposition from \cite[Eq.\ (I.3.3)]{ChOuSiYa16}. The somewhat artificial requirement that $F$ shall be an even function is linked to the finite propagation speed property. This will become apparent at the beginning of the proof of \cref{prop:reduced}.
\end{remark}

We apply \cref{prop:COSY-original} in the setting where $X=G$ is a two-step stratified Lie group and
\begin{equation}\label{eq:sub-Laplacian-iii}
L=-(X_1^2+\dots+X_{d_1}^2)
\end{equation}
is the sub-Laplacian associated with a basis $X_1,\dots,X_{d_1}$ of the first layer of the stratification $\g=\g_1\oplus \g_2$ of $G$. The measure $\mu$ in \cref{prop:COSY-original} will be the Lebesgue measure on $G$.

For our purposes of proving \cref{thm:main}, we only need \cref{cor:reduction} for the case where $L$ is a sub-Laplacian on a Heisenberg type group, but the more general version for arbitrary two-step stratified Lie groups is readily available. Note, however, that \cref{prop:COSY-original} requires the restriction type condition $(\mathrm{ST}^\infty_{p_0, 2})$ in the setting of arbitrary two-step stratified Lie groups. However, instead of using the restriction type estimates of \cref{prop:restriction}, we use a Plancherel estimate for the associated convolution kernel, which in turn implies a restriction type estimate from $L^1$ to $L^2$.

\begin{corollary}\label{cor:reduction}
Let $G$ be a two-step stratified Lie group and $L$ be a sub-Laplacian as in \cref{eq:sub-Laplacian-iii}. Let $p_{*}\in[1,2]$ and $s>1/2$. Suppose that for all $1\le p\le p_{*}$ there exists some $\varepsilon>0$ such that
\begin{equation}\label{eq:reduced}
\Vert F^{(\iota)}(\sqrt L) \Vert_{p\to p} \le C_{p,s} 2^{-\varepsilon\iota} \norm{F}_{L^2_s}
\quad \text{for all } \iota \in\N
\end{equation}
and all even bounded Borel functions $F\in L^2_s(\R)$ supported in $[-2,-1/2]\cup[1/2,2]$. Then the statements (\ref{main-(1)}) and (\ref{main-(2)}) of \cref{thm:main} hold for all $1\le p\le p_{*}$.
\end{corollary}

\begin{remark}
The assumption $1\le p\le 2(d_2+1)/(d_2+3)$ of \cref{thm:main} automatically implies that $s>1/2$ if $s>d\left(1/p-1/2\right)$ since
\[
d\bigg(\frac 1 p- \frac 1 2\bigg)\ge d_2\bigg(\frac{(d_2+3)}{2(d_2+1)}-\frac 1 2\bigg) = \frac{d_2}{d_2+1} \ge \frac 1 2.
\]
However, in \cref{cor:reduction}, we only require $1\le p\le p_{*}$ for some $p_{*}\in [1,2]$, which is why we additionally assume $s>1/2$ to make sure that $\|F|_{(0,\infty)}\|_\infty \lesssim \|F\|_{L^2_{s,\sloc}}$.
\end{remark}

\begin{proof}
The result for Bochner--Riesz multipliers in the second part of \cref{thm:main} is a direct consequence of \eqref{eq:reduced} without \cref{prop:COSY-original} involved, which can be seen in the same way as in the proof of Theorem 1.2 in \cite[Section 4]{Ni21}.

For the first part of \cref{thm:main}, we observe that
\[
\|F\|_{L^2_{s,\sloc}} \sim \|\tilde F\|_{L^2_{s,\sloc}}\quad\text{for } F(\lambda)=\tilde F(\sqrt \lambda).
\]
Thus, we may pass from $F(L)$ to the operator $F(\sqrt L)$. We use \cref{prop:COSY-original} for $p_0=1$. \Cref{prop:COSY-original} only shows that the operator $F(\sqrt L)$ is of weak type $(p,p)$, but the boundedness on $L^p$ can be easily recovered as follows: First, note that the case $p=1$ is excluded in the first part of \cref{thm:main}. Thus, suppose that the interval $(1,p_*]$ is non-empty and let $1< p \le p_*$. Suppose that $F:\R\to \C$ is a bounded Borel function satisfying
\[
\|F\|_{L^2_{s,\sloc}}<\infty \quad\text{for some }s>d\left( 1/ p - 1/ 2\right).
\]
We can choose $1<\tilde p<p$ such that $s>d\left( 1/\tilde p - 1/ 2\right)$. Applying \cref{prop:COSY-original} implies that $F(\sqrt L)$ is of weak type $(\tilde p,\tilde p)$. On the other hand, the assumption $s>1/2$ in \cref{cor:reduction} ensures that
\[
\|F|_{(0,\infty)}\|_\infty \lesssim \|F\|_{L^2_{s,\sloc}}.
\]
Hence, via interpolation with the $L^2$-$L^2$ bound provided by the spectral theorem, we may conclude that $F(\sqrt L)$ is bounded on $L^p$. The claimed estimate
\[
\|F(\sqrt L)\|_{p\to p}\le C_{p,s} \|F\|_{L^2_{s,\sloc}}
\]
in the first part of \cref{thm:main} follows by the closed graph theorem applied to the map $F\mapsto F(\sqrt L)$. Alternatively, this estimate can also be derived by inspecting the arguments of \cite{ChOuSiYa16} (see also \cite[Section 7]{Ni20}).

Now we verify the assumptions of \cref{prop:COSY-original}. Let $p_0=1$. The finite propagation speed property in (i) holds due to \cref{lem:finite-prop}, and the estimate in (iii) is automatically fulfilled by Theorem I.5 of \cite{ChOuSiYa16}. Since $p_0=1$, the restriction type condition $(\mathrm{ST}^\infty_{p_0, 2})$ is a consequence of a Plancherel estimate for the associated convolution kernel, see also \cite[Section III.5]{ChOuSiYa16}. More precisely, given a bounded Borel measurable function $F:\R\to\C$, since $L$ is a left-invariant operator, there is a convolution kernel $\mathcal{K}_{F(\sqrt L)}$ such that $F(\sqrt L) f = f * \mathcal{K}_{F(\sqrt L)}$ for all $f \in \mathcal S(\R)$. By \cite[Proposition 2]{Ch91}, there is some constant $C>0$ such that we have the Plancherel estimate
\[
\|\mathcal{K}_{F(\sqrt L)}\|_{L^2(G)}^2 = C \int_0^{\infty}|F(\sqrt\lambda)|^2 \lambda^{Q/2} \frac{d \lambda}{\lambda},
\]
where $Q$ is the homogeneous dimension of $G$. If $R>0$ and $F:\R\to\C$ is a bounded Borel function supported in $[0,R]$, then
\[
\int_0^{\infty}|F(\sqrt \lambda)|^2 \lambda^{Q/2} \frac{d \lambda}{\lambda} \lesssim R^{Q} \|F\|_\infty^2.
\]
Thus, we obtain
\[
\|F(\sqrt L)f\|_2 = \|f * \mathcal{K}_{F(\sqrt L)}\|_2 \le \|f\|_1 \|\mathcal{K}_{F(\sqrt L)}\|_2 \lesssim R^{Q/2} \|F\|_\infty \|f\|_1.
\]
By \cref{{eq:cc-volume}}, we have $|B_r^{d_{\mathrm{CC}}}(x,u)| = r^{Q} |B_1^{d_{\mathrm{CC}}}(0)|$. Hence, for $p_0=1$, the factor on the right-hand side of \cref{eq:ST-condition} is given by
\[
|B_r(x,u)|^{-1/2} ( Rr )^{Q/2} \sim R^{Q/2},
\]
which verifies the restriction type condition $(\mathrm{ST}^\infty_{p_0, 2})$. Thus the assumptions of \cref{prop:COSY-original} are satisfied.

Now suppose that the dyadic estimate \cref{eq:reduced} of \cref{cor:reduction} holds. Let $F:\R\to \C$ be a bounded Borel function such that
\[
\|F\|_{L^2_{s,\sloc}}<\infty \quad\text{for some }s>d\left( 1/ p - 1/ 2\right).
\]
To show that $F(\sqrt L)$ is of weak type $(p,p)$, we verify the required estimate \cref{eq:cond-1} of \cref{prop:COSY-original}. Note that we may assume without loss of generality that $F$ is an even function since $L$ is a positive operator. For $i\in\Z$, we let $F_i:=  F \chi_i$, where $(\chi_i)_{i\in\Z}$ is the dyadic decomposition from \cref{eq:dyadic}. Given $i,j\in \Z$, let $\iota:=i+j$ and
\[
G(\lambda):=F(2^i\lambda)\chi(\lambda),\quad\lambda\in\R,
\]
where $\chi$ is the bump function from \cref{eq:dyadic}. Then $G$ is an even function, and
\begin{align*}
(F_i)^{(j)}(\lambda)
& = (\widehat{F_i} \chi_j)^\vee(\lambda)
 = (2^i\hat G(2^i\cdot) \chi_j)^\vee(\lambda) \notag \\
& = (\hat G \chi_\iota)^\vee(2^{-i}\lambda)
 = G^{(\iota)}(2^{-i}\lambda).
\end{align*}
Let again $\delta_R$ be the dilation given by $\delta_R(x,u)=(Rx,R^2u)$. Then
\[
(F(\sqrt{L})f)\circ \delta_{R^{-1}} = F(R\sqrt L)(f \circ\delta_{R^{-1}}).
\]
This implies
\[
\norm{G^{(\iota)}(2^{-i}\sqrt L)}_{p\to p} = \norm{G^{(\iota)}(\sqrt L)}_{p\to p}.
\]
Hence, for $\iota \ge 0$, \cref{eq:reduced} yields
\begin{align*}
\norm{(F\chi_i)^{(j)}(\sqrt L)}_{ p\to  p}
& 
 = \norm{G^{(\iota)}(\sqrt L)}_{ p\to  p} \notag \\
& \lesssim 2^{-\varepsilon\iota} \Vert G^{(\iota)}\Vert_{L_s^2}
  \lesssim 2^{-\varepsilon\iota} \|F\|_{L^2_{s,\sloc}}.
\end{align*}
The case $\iota<0$ can be treated by the Mikhlin--Hörmander type result of \cite{Ch91} and \cite{MaMe90}. Suppose $\iota<0$. Let $\psi:=\sum_{i\le 2} \chi_i$. Then $\psi$ is supported in $[-8,8]$. We decompose $G^{(\iota)}$ as $G^{(\iota)}=G^{(\iota)}\psi + G^{(\iota)}(1-\psi)$. Since $G^{(\iota)}=G*\check \chi_\iota$, $\supp G\subseteq [-2,2]$ and $\check\chi\in\S(\R)$, we have
\begin{align}
\Big|\Big(\frac{d}{d\lambda}\Big)^\alpha  G^{(\iota)} (\lambda )\Big|
 & = \Big|\Big(\frac{d}{d\lambda}\Big)^\alpha \int_{-2}^2 2^\iota G(\tau) \check \chi(2^\iota(\lambda-\tau))\, d\tau  \Big| \notag \\
 & \lesssim_N 2^{\iota(\alpha+1)} \int_{-2}^2 \frac{|G(\tau)|}{(1+2^\iota|\lambda-\tau|)^N}\, d\tau,
 \quad \alpha\in\N. \label{eq:error-conv}
\end{align}
Let $Q$ denote again the homogeneous dimension of $G$. Choosing $N:=0$ in \cref{eq:error-conv} and using $2^{\iota(\alpha+1)}\le 1$, we obtain
\[
\|G^{(\iota)}\psi\|_{L^2_{Q/2+1,\sloc}}
\lesssim_{\psi} \norm{G}_2 \lesssim \|F\|_{L^2_{s,\sloc}}.
\]
On the other hand, choosing $N:=\alpha+1$ in \cref{eq:error-conv} yields in particular
\[
\Big|\Big(\frac{d}{d\lambda}\Big)^\alpha  G^{(\iota)} (\lambda )\Big| \lesssim |\lambda|^{-\alpha} \norm{G}_2 \quad \text{for } |\lambda|\ge 4.
\]
Since all derivatives of $1-\psi$ are Schwartz functions, Leibniz rule yields
\[
\norm{G^{(\iota)}(1-\psi)}_{L^2_{Q/2+1,\sloc}} \lesssim_{\psi} \norm{G}_2 \lesssim \|F\|_{L^2_{s,\sloc}}.
\]
Hence, applying the Mikhlin--Hörmander type result of \cite{Ch91} and \cite{MaMe90} yields
\begin{align*}
\norm{(F_i)^{(j)}(\sqrt L)}_{p\to p}
& = \norm{G^{(\iota)}(2^{-j}\sqrt L)}_{p\to p} \\
& = \norm{G^{(\iota)}(\sqrt L)}_{p\to p} \lesssim \|F\|_{L^2_{s,\sloc}}.
\end{align*}
This establishes the required condition \cref{eq:cond-1} of \cref{prop:COSY-original}. Thus we can apply \cref{prop:COSY-original} and we get that the operator $F(\sqrt L)$ is of weak type $(p,p)$.
\end{proof}

\section{Proof of the reduction of \texorpdfstring{\cref{thm:main}}{Theorem 1.1}}
\label{sec:main-proof}

Now suppose that $L=(X_1^2+\dots+X_{d_1}^2)$ is a sub-Laplacian on a Heisenberg type group $G$, where $X_1,\dots,X_{d_1}$ is an orthonormal basis of the first layer $\g_1$ of the stratification $\g=\g_1\oplus\g_2$. Let again $d_{\mathrm{CC}}$ denote the Carnot--Carathéodory distance associated with the vector fields $X_1,\dots,X_{d_1}$, let $d=d_1+d_2$ be the topological dimension, and $Q=d_1+2d_2$ be the homogeneous dimension of $G$. By \cref{cor:reduction}, the proof of \cref{thm:main} can be reduced to proving the statement. Given a suitable multiplier $F:\R\to\C$, we write again
\[
F^{(\iota)} := (\hat F \chi_\iota)^\vee\quad \text{for }\iota\in\Z,
\]
where $(\chi_\iota)_{\iota \in\Z}$ is the dyadic decomposition of \cref{eq:dyadic}.

\begin{proposition}\label{prop:reduced}
Suppose that $1\le p\le 2(d_2+1)/(d_2+3)$. If $s>d\left( 1/p - 1/2\right)$, then there exists some $\varepsilon>0$ such that
\[
\Vert F^{(\iota)}(\sqrt L) \Vert_{p\to p} \le C_{p,s} 2^{-\varepsilon\iota} \norm{F^{(\iota)}}_{L^2_s}\quad \text{for all } \iota \in\N
\]
and any even bounded Borel function $F\in L^2_s$ supported in $[-2,-1/2]\cup [1/2,2]$.
\end{proposition}

\begin{proof}
Let $\iota\in\N$ and $R:=2^\iota$. We proceed in several steps.

\smallskip

(1) \textit{Reduction to compactly supported functions.} Let $f\in D(G)$ be an integrable simple function on $G$. We first show that we may restrict to the case where $f$ is supported in $B_R^{d_{\mathrm{CC}}}(0)$. Since the metric space $(G,{d_{\mathrm{CC}}})$ endowed with the Lebesgue measure is a space of homogeneous type and separable, we may choose a decomposition into disjoint sets $B_j \subseteq B_R^{d_{\mathrm{CC}}}(x^{(j)},u^{(j)})$, $j\in\N$, $(x^{(j)},u^{(j)})\in G$ such that for every $\lambda\ge 1$, the number of overlapping dilated balls $B_{\lambda R}^{d_{\mathrm{CC}}}(x^{(j)},u^{(j)})$ is bounded by a constant $C(\lambda)\sim \lambda^Q$, which is independent of $\iota$. We decompose $f$ as
\[
f = \sum_{j=0}^\infty f_j\quad \text{where } f_j:=f|_{B_j}.
\]
Since $F$ is even, so is $\hat F$. As $\chi_\iota$ is even as well, the Fourier inversion formula provides
\[
F^{(\iota)} ( \sqrt L)f_j = \frac{1}{2\pi} \int_{2^{\iota-1}\le |\tau|\le 2^{\iota+1}} \chi_\iota(\tau) \hat F(\tau) \cos(\tau \sqrt L)f_j \,d\tau.
\]
Since $L$ satisfies the finite propagation speed property, $F^{(\iota)} ( \sqrt L)f_j$ is supported in $B_{3R}^{d_{\mathrm{CC}}}(x^{(j)},u^{(j)})$ by the formula above. Together with the bounded overlap of these balls, we obtain
\[
\norm{F^{(\iota)} ( \sqrt L)f}_p^p \lesssim \sum_{j=0}^\infty \Vert F^{(\iota)} ( \sqrt L)f_j \Vert_p^p.
\]
Altogether, since $L$ is left-invariant, it suffices to show
\begin{equation}\label{eq:reduction-1}
\big\| \textbf{1}_{B_{3R}^{d_{\mathrm{CC}}}(0)} F^{(\iota)}(\sqrt L) f \big\|_p \lesssim 2^{-\varepsilon\iota} \norm{F^{(\iota)}}_{L^2_s} \norm{f}_p
\end{equation}
whenever our initial function $f\in D(G)$ is supported in $B_R^{d_{\mathrm{CC}}}(0)$.

\smallskip

(2) \textit{Localizing the multiplier.} Next we show that we may replace the multiplier $F^{(\iota)}$ by $F^{(\iota)}\psi$, where $\psi$ is a smooth cut-off function which is compactly supported away from the origin. Using the dyadic decomposition $(\chi_\iota)_{\iota \in\Z}$ from \cref{eq:dyadic}, we put
\[
\psi:=\sum_{j=-2}^2 \chi_j.
\]
Then $1/8\le |\lambda| \le 8$ whenever $\lambda\in\supp\psi$, and $|\lambda| \notin (1/4,4)$ if $\lambda\in\supp(1-\psi)$. We decompose $F^{(\iota)}$ as
\[
F^{(\iota)}=F^{(\iota)}\psi + F^{(\iota)}(1-\psi).
\]
The second part of this decomposition can be treated by the Mikhlin--Hörmander type result of \cite{Ch91} and \cite{MaMe90}. Note that $F^{(\iota)}=F*\check \chi_\iota$, and $\check\chi\in\S(\R)$. Thus, given $\alpha\in\N$ and $N\in\N$, we have
\begin{align}
\Big|\Big(\frac{d}{d\lambda}\Big)^\alpha  F^{(\iota)} (\lambda )\Big|
 & = \Big|\Big(\frac{d}{d\lambda}\Big)^\alpha \int_{-2}^2 2^\iota F(\tau) \check \chi(2^\iota(\lambda-\tau))\, d\tau  \Big| \notag \\
 & \lesssim_N 2^{\iota(\alpha+1)} \int_{-2}^2 \frac{|F(\tau)|}{(1+2^\iota|\lambda-\tau|)^N}\, d\tau. \label{eq:error-conv-i}
\end{align}
Since $F$ is supported in $[-2,-1/2]\cup [1/2,2]$, choosing $N:=\alpha+2$ in \cref{eq:error-conv-i} gives
\[
\Big|\Big(\frac{d}{d\lambda}\Big)^\alpha  F^{(\iota)} (\lambda )\Big| \lesssim  2^{-\iota} \min\{|\lambda|^{-\alpha},1\} \norm{F}_2 \quad \text{whenever } |\lambda|\notin (1/4,4).
\]
This implies
\[
\norm{F^{(\iota)}(1-\psi)}_{L^2_{Q/2+1,\sloc}} \lesssim_{\psi} 2^{-\iota} \norm{F}_2.
\]
Hence, the Mikhlin--Hörmander type result of \cite{Ch91} and \cite{MaMe90} yields
\[
\Vert (F^{(\iota)}(1-\psi))(\sqrt L)\Vert_{p\to p}
 \lesssim 2^{-\iota} \norm{F}_2.
\]
Thus, instead of \cref{eq:reduction-1}, we are left proving
\begin{equation}
\big\| \textbf{1}_{B_{3R}^{d_{\mathrm{CC}}}(0)} (F^{(\iota)}\psi)(\sqrt L) f \big\|_p \lesssim 2^{-\varepsilon\iota} \norm{F^{(\iota)}}_{L^2_s} \norm{f}_p \label{eq:reduction-2}
\end{equation}
for all $f\in D(G)$ that are supported in $B_R^{d_{\mathrm{CC}}}(0)$.

\smallskip

(3) \textit{Truncation along the spectrum of $U$.} Next we decompose the operator $(F^{(\iota)}\psi)(\sqrt L)$ by a dyadic decomposition of $U$. For $\ell\in\N$, let the function $F_\ell^{(\iota)} : \R\times \R \to \C$ be given by
\[
F_\ell^{(\iota)}(\lambda,\rho) = (F^{(\iota)}\psi) (\sqrt \lambda)\chi_\ell(\lambda/\rho)\quad\text{for }\lambda\ge 0,\rho\neq 0
\]
and $F_\ell^{(\iota)}(\lambda,\rho)=0$ else.
We decompose the function on the left-hand side of \cref{eq:reduction-2} as
\begin{align}
\textbf{1}_{B_{3R}^{d_{\mathrm{CC}}}(0)} (F^{(\iota)}\psi)(\sqrt L) f
& = \textbf{1}_{B_{3R}^{d_{\mathrm{CC}}}(0)} \bigg(\sum_{\ell = 0}^\iota + \sum_{\ell = \iota + 1}^\infty \bigg) F_\ell^{(\iota)}(L,U) f \notag\\
& =: g_{\le \iota} + g_{> \iota}.\label{eq:decomp-trunc}
\end{align}
The sum over $\ell>\iota$ can be treated directly by the restriction type estimate of \cref{prop:restriction}. Recall that $|B_R^{d_{\mathrm{CC}}}(0)| \sim R^Q$ by \cref{eq:cc-volume}. Hence, Hölder's inequality with $1/q=1/p-1/2$ and the restriction type estimate \cref{eq:restriction} imply
\begin{align*}
\Vert  g_{> \iota}\Vert_p
 & \lesssim R^{Q/q} \Vert g_{> \iota}\Vert_2 \\
 & \le R^{Q/q}\bigg\Vert \sum_{\ell = \iota + 1}^\infty F_\ell^{(\iota)}(L,U) f \bigg\Vert_2\\
 & \lesssim 2^{\iota(Q-d_2)/q} \norm{F^{(\iota)}\psi}_2 \norm{f}_p \\
 & \lesssim_\psi 2^{-\varepsilon\iota}  \norm{F^{(\iota)}}_{L^2_s} \norm{f}_p
\end{align*}
if we choose $0<\varepsilon<s-d/q$. (Note that $\|F^{(\iota)}\|_2\sim 2^{-\iota s}\|F^{(\iota)}\|_{L^2_s}$ due to the localization in frequency.) Thus, we are done once we have also treated the sum over $\ell\in\{-1,\dots, \iota\}$, that is, it remains to show
\begin{equation}\label{eq:small-eigenvalues}
\Vert g_{\le \iota} \Vert_p
 \lesssim  2^{-\varepsilon\iota} \norm{F^{(\iota)}}_{L^2_s} \norm{f}_p.
\end{equation}

\smallskip

(4) \textit{The support of the convolution kernel.}
Let $\mathcal K^{(\iota)}_\ell$ be the convolution kernel of the operator $F^{(\iota)}_\ell(L,U)$. By \cref{eq:cc-equivalence}, there is a constant $C>0$ such that
\[
B_R^{d_{\mathrm{CC}}}(0)
 \subseteq B_{CR}(0) \times B_{CR^2}(0).
\]
Hence the function $f$ is supported in a Euclidean ball of dimension $R\times R^2$. In view of the finite propagation speed property which we exploited in part (1) of the proof, we may think of $\mathcal K^{(\iota)}_\ell$ being supported in a ball of dimension $R\times R^2$ as well (which is of course not quite true since we replaced the multiplier $F^{(\iota)}$ by $F^{(\iota)}\psi$). In the following, we show that the convolution kernel $\mathcal K^{(\iota)}_\ell$ of the truncated multiplier is essentially supported in an even smaller ball of dimension $R_\ell R^\gamma \times R^2$, where $R_\ell:=2^\ell$ and $\gamma>0$ will be a number chosen sufficiently small, depending only on the parameters $s,p,d_1,d_2$. For convenience, we introduce the following notation: We will write
\[
A\lesssim_\iota B
\]
whenever $A\le R^{C(p,d_1,d_2)\gamma} B$ for some constant $C(p,d_1,d_2)>0$ depending only on the parameters $p,d_1,d_2$.

Given $\ell \in\{0,\dots,\iota\}$, we split the Euclidean ball $B_{CR}(0) \times B_{CR^2}(0)$ into a grid with respect to the first layer, which gives a decomposition of $\supp f\subseteq B_R^{d_{\mathrm{CC}}}(0)$ such that
\[
\supp f = \bigcup_{m=1}^{M_{\ell}} B_{m}^{(\ell)},
\]
where $B_{m}^{(\ell)} \subseteq B_{CR_\ell}(x_{m}^{(\ell)}) \times B_{CR^2}(0)$ are disjoint subsets, and $|x_{m}^{(\ell)} - x_{m'}^{(\ell)}| > R_\ell/2$ for $m\neq m'$. Then the number $M_\ell$ of balls in this decomposition is bounded by 
\begin{equation}\label{eq:number-balls}
M_\ell\lesssim (R/R_\ell)^{d_1} = 2^{d_1(\iota-\ell)}.
\end{equation}
Moreover, given $\gamma>0$, the number of overlapping balls
\[
\tilde B_{m}^{(\ell)}:=B_{2 C R_\ell R^\gamma}(x_{m}^{(\ell)}) \times B_{9CR^2}(0),\quad 1\le m\le M_{\ell}
\] 
can be bounded by a constant $N_\gamma \lesssim_\iota 1$ (which is independent of $\ell$). We decompose the function $f$ as
\[
f = \sum_{m=1}^{M_{\ell}} f|_{B_{m}^{(\ell)}}.
\]
In the following, we show that the function
\[
g_{m}^{(\ell)}:= \textbf{1}_{B_{3R}^{d_{\mathrm{CC}}}(0)} F_\ell^{(\iota)}(L,U) (f|_{B_{m}^{(\ell)}})
\]
is essentially supported in the ball $\tilde B_{m}^{(\ell)}$. We decompose the function $g_{\le \iota}$ of \cref{eq:decomp-trunc} as
\begin{equation}\label{eq:decomp-ess-supp}
g_{\le \iota} = \sum_{\ell=0}^\iota \sum_{m=1}^{M_{\ell}} g_{m}^{(\ell)}|_{\tilde B_{m}^{(\ell)}} + \sum_{\ell=0}^\iota \sum_{m=1}^{M_{\ell}} g_{m}^{(\ell)}|_{\g\setminus\tilde B_{m}^{(\ell)}}
=: g_{\le \iota}^{(1)} + g_{\le \iota}^{(2)}.
\end{equation}
To show that the second summand is negligible (in the sense of \cref{eq:error-Lp}), we interpolate between $L^1$ and $L^2$ via the Riesz--Thorin interpolation theorem. For the $L^1$-$L^1$ estimate, note that $(x,u)\in (\g\setminus\tilde B_{m}^{(\ell)})\cap B_{3R}^{d_{\mathrm{CC}}}(0)$ and $(x',u')\in B_{m}^{(\ell)}$ imply
\[
|x-x'|\ge CR_\ell R^\gamma.
\]
Let $\mathcal K_\ell^{(\iota)}$ be the convolution kernel associated with $F_\ell^{(\iota)}(L,U)$. Then
\[
F_\ell^{(\iota)}(L,U)(f|_{B_{m}^{(\ell)}})(x,u) = (f|_{B_{m}^{(\ell)}}) * \mathcal K_\ell^{(\iota)}(x,u),
\]
and we obtain
\begin{align}
\norm{g_{\le \iota}^{(2)}}_1
& \le \int_G \sum_{\ell=0}^\iota \sum_{m=1}^{M_{\ell}} \textbf{1}_{(\g\setminus\tilde B_{m}^{(\ell)}) \cap B_{3R}^{d_{\mathrm{CC}}}(0)}(x,u)\,\big| (f|_{B_{m}^{(\ell)}}) * \mathcal  K_\ell^{(\iota)}(x,u)\big| \,d(x,u) \notag \\
& \le \int_{B_R^{d_{\mathrm{CC}}}(0)}  \sum_{\ell=0}^\iota \int_{A_\ell(x')} \sum_{m=1}^{M_{\ell}} \big|f|_{B_{m}^{(\ell)}}(x',u')\big| \notag \\
& \hspace{3cm} \times \big|\mathcal K_\ell^{(\iota)}\big((x',u')^{-1}(x,u)\big)\big| \,d(x,u)\,d(x',u') \notag \\
& = \int_{B_R^{d_{\mathrm{CC}}}(0)} |f(x',u')| \,\kappa_\gamma(x',u') \,d(x',u'), \label{eq:error-L1-2}
\end{align}
where
\[
A_\ell(x') := \{ (x,u)\in B_{3R}^{d_{\mathrm{CC}}}(0) : |x-x'|\ge CR_\ell R^\gamma \}
\]
and
\[
\kappa_\gamma(x',u') := \sum_{\ell=0}^\iota  \int_{A_\ell(x')} \big|\mathcal K_\ell^{(\iota)}\big((x',u')^{-1}(x,u)\big) \big| \,d(x,u).
\]
Given $N\in\N$, the Cauchy--Schwarz inequality yields
\begin{align}
\int_{A_\ell(x')} & |\mathcal K_\ell^{(\iota)}\big((x',u')^{-1}(x,u)\big) | \,d(x,u) \notag \\
& \lesssim (R_\ell R^\gamma)^{-N} \int_{A_\ell(x')} \big||x-x'|^N \mathcal K_\ell^{(\iota)}(x-x',u-u'- \tfrac 1 2 [x',x] ) \big| \,d(x,u) \notag \\
& \lesssim (R_\ell R^\gamma)^{-N} R^{Q/2} \bigg( \int_G \big||x|^N \mathcal K_\ell^{(\iota)}(x,u) \big|^2 \,d(x,u) \bigg)^{1/2}. \label{eq:error-CS}
\end{align}
In the last line we used again that $|B_{3R}^{d_{\mathrm{CC}}}(0)|\sim R^{ Q}$ by \cref{eq:cc-volume}. 
By \cref{prop:weighted-plancherel}, the second factor of \cref{eq:error-CS} can be estimated by
\[
\int_G \big||x|^N \mathcal K_\ell^{(\iota)}(x,u) \big|^2 \,d(x,u)
 \lesssim_N R_\ell^{2N-d_2} \norm{F^{(\iota)}}_2^2.
\]
Hence
\[
\kappa_\gamma(x',u')
\lesssim_N \sum_{\ell=0}^\iota R^{-\gamma N + Q/2} R_\ell^{- d_2/2} \norm{F^{(\iota)}}_2
\lesssim R^{-\gamma N + Q/2} \norm{F^{(\iota)}}_2.
\]
Altogether, with \cref{eq:error-L1-2}, we have
\begin{equation}\label{eq:error-L1}
\norm{g_{\le \iota}^{(2)}}_1 \lesssim_N R^{-\gamma N + Q/2} \norm{F^{(\iota)}}_2 \norm{f}_1.
\end{equation}
For the $L^2$-$L^2$ estimate, we use the trivial estimate
\begin{equation}\label{eq:error-L2-triangle}
\norm{g_{\le \iota}^{(2)}}_2
 = \bigg\Vert \sum_{\ell=0}^\iota \sum_{m=1}^{M_{\ell}} g_{m}^{(\ell)}|_{\g\setminus\tilde B_{m}^{(\ell)}} \bigg\Vert_2
 \le \sum_{\ell=0}^\iota \sum_{m=1}^{M_{\ell}} \Vert g_{m}^{(\ell)}\Vert_2.
\end{equation}
Since $\norm{\chi_\ell}_\infty\le 1$, each summand of \cref{eq:error-L2-triangle} can be estimated by
\[
\norm{g_{m}^{(\ell)}}_2 \le \norm{ F_\ell^{(\iota)}(L,U) (f|_{B_{m}^{(\ell)}})}_2
\le \norm{ F^{(\iota)} }_\infty \norm{f|_{B_{m}^{(\ell)}}}_2.
\]
Using Hölder's inequality on the right-hand side of \cref{eq:error-L2-triangle} yields
\[
\norm{g_{\le \iota}^{(2)}}_2
 \le \norm{F^{(\iota)}}_\infty (\iota+1) M_{\ell}^{1/2} \norm{f}_2.
\]
Together with \cref{eq:number-balls} and the Sobolev embedding
\[
\norm{F^{(\iota)}}_\infty \lesssim \norm{F^{(\iota)}}_{L^2_{\alpha}}
\sim R^{\alpha} \norm{F^{(\iota)}}_2,\quad \alpha > 1/2,
\]
we obtain
\begin{align}
\norm{g_{\le \iota}^{(2)}}_2
 & \lesssim R^{\alpha} \norm{F^{(\iota)}}_2 (\iota+1)M_{\ell}^{1/2} \norm{f}_2 \notag \\
 & \lesssim R^{\alpha'} \norm{F^{(\iota)}}_2 \norm{f}_2\quad\text{for } \alpha'>\alpha+d_1/2. \label{eq:error-L2}
\end{align}
Applying the Riesz--Thorin interpolation theorem with \cref{eq:error-L1} and \cref{eq:error-L2} and choosing $N=N(\gamma)\in \N$ sufficiently large in \cref{eq:error-L1} yields
\begin{equation}\label{eq:error-Lp}
\norm{g_{\le \iota}^{(2)}}_p
\lesssim_N 2^{-\iota N} \norm{F^{(\iota)}}_2 \norm{f}_p.
\end{equation}
In view of the decomposition \cref{eq:decomp-ess-supp}, for showing \cref{eq:small-eigenvalues}, it thus remains to prove
\begin{equation}\label{eq:main-term}
\Vert g_{\le \iota}^{(1)} \Vert_p
 = \bigg\Vert \sum_{\ell=0}^\iota \sum_{m=1}^{M_{\ell}} \tilde g_{m}^{(\ell)} \bigg\Vert_p
 \lesssim  2^{-\varepsilon\iota} \norm{F^{(\iota)}}_{L^2_s} \norm{f}_p,
\end{equation}
where
\[
\tilde g_{m}^{(\ell)}:= \textbf{1}_{\tilde B_m^{(\ell)}\cap B_{3R}^{d_{\mathrm{CC}}}(0)} F_\ell^{(\iota)}(L,U) (f|_{B_{m}^{(\ell)}}).
\]
On a formal level, this means that we may indeed assume that the convolution kernel $\mathcal K_\ell^{(\iota)}$ is supported in a ball of dimension $R_\ell R^\gamma\times R^2$.

\smallskip

(5) \textit{The main contribution.} Hölder's inequality and the bounded overlapping property of the balls $\tilde B_{m}^{(\ell)}$ imply
\begin{equation}\label{eq:main-term-overlap}
\norm{g_{\le \iota}^{(1)}}_p^p
 \lesssim_\iota (\iota+1)^{p-1} \sum_{\ell=0}^\iota \sum_{m=1}^{M_{\ell}} \Vert\tilde g_{m}^{(\ell)}\Vert_p^p.
\end{equation}
Using Hölder's inequality together with the restriction type estimate \cref{eq:restriction} yields
\begin{align*}
\Vert\tilde g_{m}^{(\ell)}\Vert_p
& \lesssim ( (R_\ell R^\gamma)^{d_1} R^{2d_2})^{1/q} \norm{g_{m}^{(\ell)}}_2 \\
& \lesssim_\iota (R_\ell^{d_1} R^{2 d_2})^{1/q} \norm{g_{m}^{(\ell)}}_2 \\
& \lesssim (R_\ell^{d_1-d_2} R^{2 d_2 })^{1/q} \norm{F^{(\iota)}}_2 \norm{f|_{B_{m}^{(\ell)}}}_p. 
\end{align*}
Plugging this estimate into the right-hand side of \cref{eq:main-term-overlap} and using the fact that the functions $f|_{B_{m}^{(\ell)}}$ have disjoint support, we obtain
\[
\norm{g_{\le \iota}^{(1)}}_p^p \lesssim_\iota (\iota+1)^{p-1} \sum_{\ell=0}^\iota (R_\ell^{d_1-d_2} R^{2 d_2 })^{p/q} \norm{F^{(\iota)}}_2^p \norm{f}_p^p.
\]
Choosing $\gamma>0$ small enough, we may conclude that
\begin{align*}
\norm{g_{\le \iota}^{(1)}}_p^p
\lesssim 2^{- \varepsilon \iota} \sum_{\ell=0}^\iota (R_\ell^{d_1-d_2} R^{d_2-d_1})^{p/q} \norm{F^{(\iota)}}_{L^2_s}^p  \norm{f}_p^p
\end{align*}
for some $0<\varepsilon<s-d/q$. Recall that $R_\ell =2^\ell$ and $R=2^\iota$. Since $d_1> d_2$, we have
\[
 \sum_{\ell=0}^\iota \Big(\Big(\frac{R_\ell}{R}\Big)^{d_1-d_2}\Big)^{p/q}
 \le C_p.
\]
Altogether, we obtain
\[
\norm{g_{\le \iota}^{(1)}}_p
 \lesssim 2^{- \tilde\varepsilon \iota} \norm{F^{(\iota)}}_{L^2_s} \norm{f}_p
\]
for some $\tilde\varepsilon>0$. This is \cref{eq:main-term}, so the proof is concluded.
\end{proof}

\section{Remarks on weighted restriction type estimates for sub-Laplacians}
\label{sec:weighted}

In \cite{ChOu16}, Chen and Ouhabaz proved a spectral multiplier theorem for the Grushin operator $\G=-\Delta_x-|x|^2\Delta_u$ acting on $\R^{d_1}\times \R^{d_2}$ by using a weighted restriction type estimate of the form
\[
\norm{|x|^\alpha F(\sqrt \mathcal G) f}_{L^2(\R^{d_1}\times \R^{d_2})} \le C_{p,\alpha} \norm{F}_{L^2(\R)} \norm{f}_{L^p(\R^{d_1}\times \R^{d_2})},
\]
where $\alpha>0$ and $F:\R\to \C$ is a bounded Borel function supported in $[1/4,1]$. Let $L$ denote again a sub-Laplacian on a Heisenberg type group $G$ with Lie algebra $\mathfrak g=\mathfrak g_1\oplus\mathfrak g_2$ with layers of dimension $d_1$ and $d_2$, respectively. Then
\[
L f = \mathcal G f
\]
for any $\mathfrak g_1$-radial function on the Heisenberg type group $G$, i.e., a function on $G$ which only depends on $|x|$ (with $x\in\mathfrak g_1$) and $u\in\mathfrak g_2$ (where we identify again $G$ with its Lie algebra $\mathfrak g$, which is in turn identified with $\R^{d_1}\times \R^{d_2}$). In view of this close relationship, one might hope that the approach of Chen and Ouhabaz can also be applied in the setting of Heisenberg type groups. However, a crucial ingredient of their approach is the sub-elliptic estimate
\begin{equation}\label{eq:sub-elliptic-2}
\norm{|x|^\alpha g}_{L^2(\R^{d_1})} \le C_\alpha \norm{ |\mu|^{-\alpha} (H^\mu)^{\alpha/2} g}_{L^2(\R^{d_1})},\quad g\in L^2(\R^{d_1}),
\end{equation}
where $H^\mu =- \Delta_x  + \tfrac 1 4 |x|^2 |\mu|^2$ denotes again the rescaled Hermite operator on $\R^{2n}$. Unfortunately, the analogous estimate of \cref{eq:sub-elliptic-2} in our setting, where $H^\mu$ is replaced by the $\mu$-twisted Laplacian $L^\mu$ of \cref{eq:twisted-laplace}, fails. We will prove in the following that the estimate \cref{eq:sub-elliptic-2} where $H^\mu$ is replaced by $L^\mu$ is false for $\alpha=1$. (The approach of \cite{ChOu16} requires to choose $0<\alpha<d_2(1/p-1/2)$ as large as possible, so large values of $\alpha$ are the crucial ones.) 

Via \cref{eq:rotation-twisted} and a linear substitution, the estimate 
\[
\norm{|x|^\alpha g}_{L^2(\g_1)} \le C_\alpha \norm{ |\mu|^{-\alpha} (L^\mu)^{\alpha/2} g}_{L^2(\g_1)}
\]
is equivalent to
\[
\norm{|z|^\alpha g}_{L^2(\R^{2n})} \le C_\alpha \norm{ |\mu|^{-\alpha} (L_0^{|\mu|})^{\alpha/2} g}_{L^2(\R^{2n})},
\]
where $L_0^{|\mu|}$ is the twisted Laplacian of \cref{eq:classical-twisted}. Rescaling with $|\mu|$, we may restrict to the case $|\mu|=1$. Let $A:=L_0^1$. Then, by \cref{eq:classical-twisted},
\[
A = -\Delta_z + \tfrac 1 4 |z|^2 - i N,
\]
where, when writing $z=(a_1,\dots,a_n,b_1,\dots,b_n)$,
\[
N = \sum_{j=1}^n (a_j \partial_{b_j} - b_j \partial_{a_j}).
\]
Now, suppose that
\begin{equation}\label{false-sub-elliptic-1}
\norm{|z| g}_{L^2(\R^{2n})} \le C \norm{A^{1/2} g}_{L^2(\R^{2n})} \quad\text{for all } g\in \S(\R^{2n}).
\end{equation}
Recall that the matrix coefficients $\Phi_{\nu,\nu'}$ of the Schrödinger representation $\pi_1(\cdot,0)$ given by \cref{eq:special-hermite} are eigenfunctions of $A$ with
\begin{equation}\label{eq:eigenvalue-1}
A \Phi_{\nu,\nu'} = (2|\nu'|_1+n)\Phi_{\nu,\nu'}.
\end{equation}
On the other hand, the functions $\Phi_{\nu,\nu'}$ are also eigenfunctions of $H=-\Delta_z + \tfrac 1 4 |z|^2$ by Equation~(1.3.25) of \cite{Th93}, with
\begin{equation}\label{eq:eigenvalue-2}
H \Phi_{\nu,\nu'} = (|\nu|_1+|\nu'|_1+n)\Phi_{\nu,\nu'}.
\end{equation}
When writing $\zeta=(\alpha_1,\dots,\alpha_n,\beta_1,\dots,\beta_n)$, direct computation shows
\[
\widehat{Ag}(\zeta) = \hat A \hat g (\zeta),
\]
where the operator $\hat A$ is given by
\[
\hat A := |\zeta|^2 - \tfrac 1 4 \Delta_\zeta + i \sum_{j=1}^n ( \partial_{\alpha_j} \beta_j - \partial_{\beta_j} \alpha_j)
\]
and $\widehat\cdot$ denotes the $2n$-dimensional Fourier transform given by
\[
\hat g(\zeta) = \int_{\R^{2n}} g(z) e^{-i\langle \zeta, z\rangle_{\R^{2n}}} \, dz, \quad \zeta\in\R^{2n}.
\]
Since $\hat A(g(2\zeta))=(Ag)(2\zeta)$, the estimate \cref{false-sub-elliptic-1} together with Plancherel's theorem implies 
\begin{equation}\label{false-sub-elliptic-2}
\norm{(-\Delta_z)^{1/2} g}_{L^2(\R^{2n})} \le C \norm{A^{1/2} g}_{L^2(\R^{2n})}\quad \text{for all } g\in \S(\R^{2n}).
\end{equation}
Setting $g:=\Phi_{\nu,\nu'}$ and using \cref{eq:eigenvalue-2}, \cref{false-sub-elliptic-1}, \cref{false-sub-elliptic-2}, \cref{eq:eigenvalue-1}, we obtain
\begin{align*}
(|\nu|_1+|\nu'|_1+n)\norm{g}_{L^2(\R^{2n})}^2
 & = \norm{H^{1/2} g}_{L^2(\R^{2n})}^2 \\
 & = \big((-\Delta_z+\tfrac 1 4 |z|^2)g,g \big) \\
 & = \big\Vert (-\Delta_z)^{1/2} g\big\Vert _{L^2(\R^{2n})}^2 + \tfrac 1 4 \big\Vert |z| g\big\Vert _{L^2(\R^{2n})}^2 \\
 & \le C \norm{A^{1/2}g}_{L^2(\R^{2n})}^2 \\
 & = C (2|\nu'|_1+n) \norm{g}_{L^2(\R^{2n})}^2.
\end{align*}
Now fixing $\nu'\in\N^n$ and letting $|\nu|\to\infty$ yields a contradiction, whence the assumed estimate \cref{false-sub-elliptic-1} is indeed false.


\end{document}